\documentclass[12pt]{amsart}

\usepackage{amsmath}
\usepackage{enumitem}
\usepackage{amsthm}
\usepackage{amsfonts}
\usepackage{amssymb}
\usepackage{mathrsfs}
\usepackage{hyperref}
\usepackage[dvipsnames]{xcolor}
\usepackage[all]{hypcap}[2006/02/20]
\usepackage{fullpage}
\usepackage{esint}
\usepackage{graphicx}
\usepackage{tikz}

\DeclareMathAlphabet{\mathpzc}{OT1}{pzc}{m}{it}

\newcommand{\bb}[1]{\mathbb{#1}}

\newcommand{\pard}[2]{\frac{\partial #1}{\partial #2}}

\newcommand{\ho}{\left(\frac{d}{dt} -\Delta \right)}
\newcommand{\ddt}[1]{\frac{ d #1}{dt}}
\newcommand{\ip}[2]{\left \langle #1 , #2 \right\rangle}

\newcommand{\n}{\nabla}
\newcommand{\diam}{\operatorname{diam}}
\newcommand{\R}{\mathbb{R}}
\newcommand{\e}{\epsilon}

\renewcommand{\l}{\lambda}

\renewcommand{\a}{\alpha}

\newcommand{\ra}{\rightarrow}
\newcommand{\ov}{\overline}

\renewcommand{\t}{\theta}

\renewcommand{\div}{\text{div}}
\newcommand{\T}{\Theta}

\newcommand{\on}[1]{{\operatorname{#1}}}

\begin{document}
\theoremstyle{plain}
\newtheorem{theorem}{Theorem}
\newtheorem{lemma}[theorem]{Lemma}
\newtheorem{claim}[theorem]{Claim}
\newtheorem{proposition}[theorem]{Proposition}
\newtheorem{cor}[theorem]{Corollary}
\theoremstyle{definition}
\newtheorem{defses}{Definition}
\newtheorem{example}{Example}
\newtheorem{assumption}{Assumption}
\theoremstyle{remark}
\newtheorem{remark}{Remark}
\title{Nonlocal estimates for the Volume Preserving Mean Curvature Flow and applications}
\author{Ben Lambert}
\email{b.s.lambert@leeds.ac.uk}
\address{ School of Mathematics, University of Leeds, Leeds, LS2 9JT, United Kingdom}
\author{Elena M\"ader-Baumdicker}
\email{maeder-baumdicker@mathematik.tu-darmstadt.de}
\address{Fachbereich Mathematik, Schlossgartenstr. 7, 64289 Darmstadt, Germany}
\begin{abstract}
We obtain estimates on nonlocal quantities appearing in the Volume Preserving Mean Curvature Flow (VPMCF) in the closed, Euclidean setting. As a result we demonstrate that blowups of finite time singularities of VPMCF are ancient solutions to Mean Curvature Flow (MCF), prove that monotonicity methods may always be applied at finite times and obtain information on the asymptotics of the flow.
\end{abstract}
\maketitle

\section{Introduction}

A key question for the study of Volume Preserving Mean Curvature Flow (VPMCF) is the extent to which it has similar properties to the Mean Curvature Flow (MCF). It is now known that the extra nonlocal term in VPMCF causes significant differences.  For example, Cabezas-Rivas and Miquel \cite{EstherCounterexamples} prove that key properties used in many MCF methods do not hold for VPMCF, such as preservation of mean convexity (along with most other curvature conditions bar convexity). This demonstrates the need to a VPMCF specific methods, and in this paper we begin to develop some of these. In this paper, we give evidence that there are still significant links between singularities of these two flows: We show that blowups of type I finite time singularities of VPMCF are indeed ancient solutions to MCF (compare with \cite[point (c), p288]{EstherCounterexamples}). Finite time type~II singularities of the VPMCF look like eternal solutions of MCF after a suitable parabolic blowup procedure (first used by Hamilton~\cite{Hamilton}).  This is particularly interesting in the light of the significant progress has been made in recent years in the classification of such singularities - see for example \cite{AngenentDaskalopoulosSesum}\cite{BrendleChoi}\cite{ChoiHaslhoferHershkovits}\cite{ChoiHaslhoferHershkovitsWhite}.\\[-0.2cm]

In order to study these blowup solutions we needed to get control over the (nonlocal) Lagrange parameter in the equation. We found links between $L^2$-bounds on the (nonlocal) averaged mean curvature and the extrinsic diameter along the flow. These estimates play an important role in the monotonicity formula, and essentially allow initial results on the analysis of singularities to be proven without assuming any further bounds on $\ov{H}$. \\[-0.2cm]

The VPMCF of $n$ dimensional hypersurfaces in $\bb{R}^{n+1}$ is the $L^2$-gradient flow of the area functional under the constraint that the $(n+1)$-dimensional enclosed volume is fixed.
This flow was introduced by Gage~\cite{Gage} for curves and by Huisken~\cite{HuiskenConvexVPMCF} for hypersurfaces. We consider the closed setting, meaning that $M^n$ is a smooth, orientable, compact $n$-dimensional manifold without boundary. Then a smooth family of immersions $X:M^n\times[0,T)\ra \bb{R}^{n+1}$ is a solution of the VPMCF if 
\begin{equation*}
\left\{
\begin{aligned}
\ddt{X}& = -(H-\ov{H})\nu\ , &&\text{ on } M^n \times [0,T),\\
X(\cdot, t) &= X_0 \hspace{2,1cm} \ &&\text{ on } M^n, 
\end{aligned}\right.
\end{equation*}
where $X_0: M^n \to \bb{R}^{n+1}$ is a given immersion. Here, $\nu$ is a unit normal along $X(\cdot, t) =:M_t$,  $h_{ij} = - \langle \partial_i\partial_jX,\nu\rangle$ are the local coefficients of the second fundamental form $A$, $g_{ij} = \ip{X_i}{X_j}$ is the induced metric and $H = \sum g^{ij}h_{ij}$ is the mean curvature. Furthermore, $\ov{H} =\frac{\int H d\mu_t}{|M_t|}$ is the Lagrange parameter of the above mentioned geometric variational problem. Thus, $\frac{d}{dt}|M_t|\leq 0$ and $\frac{d}{dt}V(t)=0$ is satisfied along the flow (for the precise definition of the enclosed volume $V(t)$ see Section~\ref{sec:Notation}). Note that the formulation of the VPMCF does not depend on the choice of the unit normal. If it is possible to choose between an ``inner'' and ``outer'' unit normal with respect to $M_t$, then we choose the outer unit normal, so that $H>0$ for round spheres.\\[-0.2cm]

We collect some known results about the VPMCF. Gage (for $n=1$) and Huisken (for $n\geq 2$) proved that convex solutions of the VPMCF stay convex and do not develop singularities. They converge smoothly to a round sphere enclosing the same amount of volume as the initial hypersurface. In~\cite{EscherSimonett}, Escher and Simonett proved that hypersurfaces being $h^{1+\beta}$ close to a round sphere converge to (possibly another) round sphere with the same enclosed volume as the initial surface. Li obtained related results in \cite{Li}. For example, he proved convergence to a round sphere if the integrated trace-free second fundamental form $\overset{\circ}A_{ij} = h_{ij} - \frac{H}{n}g_{ij}$ of the initial surface is small enough in $L^2$. The needed smallness depends on  $\max|A|(0)$ , $\ov{H}(0)>0$ and $V(0)>0$, see \cite[Theorem~1]{Li}. 
Further results about the VPMCF in a non-Euclidean ambient manifold can be found for example in \cite{HuiskenYau, EckerHuisken, Alikakos, Rigger, Cabezas07, Cabezas09}.\\[-0.2cm]

Athanassenas, later with Kandanaarachchi, studied in \cite{Atha97, Atha03, AthaKanda12} the VPMCF of a rotational symmetric surface with and without boundary conditions. For example, in \cite{AthaKanda12} it is shown that a height bound on the generating curve for the rotational symmetric surfaces prevents the flow from developing a singularity. The hypersurfaces then converge to a sphere (or a half-sphere in the considered boundary setting). Note that rotational symmetry (additionally to convexity) is one the of few properties that is known to be preserved under the VPMCF.\\[-0.2cm]

For curves, we expect stronger results compared to the hypersurface case because the main part of the Lagrange parameter has a geometric interpretation -- the term $\int \kappa ds$ is scaling invariant and measures the turn of the angle of a curve. In the closed setting, this is just $2\pi m$ with $m\in\mathbb{N}$ (the index of a closed curve). So an $L^\infty$-in-time bound on $\ov{\kappa} = \frac{\int \kappa ds_t}{L_t}$ is immediate if one controls the length from below (e.g.\ by an isoperimetric inequality). Inspired by \cite{Chou}, Escher and Ito proved in \cite{EscherIto} the existence of singularities of the VPMCF for curves under some conditions on the initial curve. One scenario of the initial curve leading to a singularity is that the enclosed volume of the initial curve with $\frac{1}{2\pi}\int \kappa ds \geq 1$ is negative. Inspired by \cite{EscherIto}, the second author proved in \cite{Elena18} an analogous result for curves in the (volume preserving) free boundary setting. In this setting, $\int \kappa ds$ is not bounded a priori, but still, an $L^\infty$-in-time proven earlier in \cite{Elena15} is used in the proof of existence of singularities. In \cite{Elena15}, the second author had introduced the (Neumann) free boundary setting for the VPMCF in detail (for curves and surfaces) and proved convergence to a stationary solution for (non-closed) curves under some conditions on the initial curve. The most advanced result about closed curves can be found in \cite{Dittberner}. There, Dittberner was able to derive a comparison principle between the intrinsic and extrinsic distance -- based on previous work of Huisken~\cite{HuiskenComparison} for the MCF -- to show convergence to a sphere of a closed curve under the condition that the initial curve is embedded and satisfies $\int_{x}^y \kappa ds \geq -\pi$ for all $x,y$. This can be seen as the analogue of Grayson's theorem -- which says that the MCF first makes every embedded, closed curve convex and then shrinks it to a point while becoming more and more round, see \cite{Grayson, GageHamilton}. Without assuming $\int_{x}^y \kappa ds \geq -\pi$ for all $x,y$ we cannot expect that such a strong result holds because there is an example, already suggested by Gage \cite{Gage} and also studied by Dittberner \cite{Dittberner} where this condition is violated and a self-intersections develops after starting the VPMCF. Numerical computations in fact indicate that then a singularity appears for this initial curve, see Mayer~\cite{Mayer}. For curves, it was also shown recently, that a star-shaped, centrosymmetric set keeps this properties and converges to a round sphere, see~\cite{Gaoetal}.\\[-0.2cm]

Whether star-shapedness is perserved in general seems to be an open questions. In the work of Kim and Kwon certain approximating solutions of the VPMCF are considered. Using them, it is proven that a strong version of star-shapedness is preserved \cite{KimKwon}. In  \cite{Mugnaietal}, Mugnai, Seis and Spadaro presented a distributional formulation of the VPMCF using the setting of Caccioppoli sets. They show global-in-time existence of their weak flow. A phase field method of the VPMCF was studied for example in the work of Takasao \cite{TakasaoI, TakasaoII}, see also the references therein. In which generality these weak flows or limits of approximating sets can be applied to any initially smoothly immersed surfaces in unclear to the authors. An immersed surface does not need to bound a domain, see also the definition of Alexandrov immersions below. Recently, Laux showed in  in \cite{Laux} that a (strong) solution where the hypersurfaces bound a domain agrees with a distributional solution he defines (as long as the smooth one exists).  \\[-0.2cm]

We now point out the structure of the paper and formulate the main results. Due to the lack of preserved quantities under the VPMCF we need to work in a quite general setting. In Section~\ref{sec:diamL2}, we study the connection between (extrinsic) diameter bounds and $L^2$-estimates of $\ov{H}(t)=\frac{\int H d\mu_t}{|M_t|}$. One of our results is the following:
\begin{center}
\emph{If the initial immersion has non-vanishing enclosed volume $V_0\not=0$, then there \\ are constants $c$ and $C$ only depending on $M_0$ such that
\begin{align*}
 \int_0^t \ov{H}^2 (\tau) d\tau \leq C(1+ t^2) e^{c\sqrt{t}}.
\end{align*}
}
\end{center}
We provide an example of a convex, embedded curve where the diameter is growing for a short time (note the contrast to the classical MCF). In Section~\ref{sec:4}, we re-prove a monotonicity formula for the VPMCF already shown in \cite{ElenaDiss} and study consequences for finite time singularities. In the monotonicity formula it is important that the $L^2$-norm of $\ov{H}$ is appearing in form of $\exp(-\frac{1}{2}\int_0^t\ov{H}^2(\tau) d\tau)$ as a multiplicative factor in front of the usual integrated Gau\ss ian in Huisken's monotonicity formula \cite{HuiskenMono}. Consequently, the $L^2$-control from Section~\ref{sec:diamL2} allows us to get properties of asymptotic flows appearing after suitable parabolic blowups at finite time singularities. One result is:
\begin{center}
\emph{Parabolic blowups of VPMCF about finite time singularities of type I ($\max|A|^2 \leq \frac{C}{T-t}$) produce ancient, homothetically shrinking  solutions of MCF. Finite type II blowups are eternal solutions of the MCF after Hamilton's parabolic rescaling.}
\end{center}
In Subsection~\ref{subsec:clearingout}, we illustrate that an $L^\infty$-bound on $\ov{H}$ in fact implies that the diameter of $|M_t|$ along the VPMCF always stays uniformly bounded for all times. \\[-0.2cm]

Section~\ref{sec:longtime} contains statements about the VPMCF with infinite life-span and uniform diameter bound. We introduce a new ``extended isoperimetric ratio'' by
\begin{align*}
 \mathcal{I}(M_t): = \frac{n+1}{n}\ov{H} \frac{V(t)}{|M_t|}.
\end{align*}
Note that $\mathcal{I}$ of the round sphere is equal to one. 
We motivate this definition with the Alexandrov-Fenchel inequalities proven for $k$-convex, star-shaped domains in \cite{GuanLiQuermass}. The main result of Section~\ref{sec:longtime} is the following:
\begin{center}
\emph{If $V_0\not =0$ and $M_t$ satisfies a uniform diameter bound along the VPMCF with $T=\infty$, then, for each $t_i\to\infty$, there exists a subsequence (not relabeled) such that \\ either $\mathcal{I}(M_{t_i})\to 0$ or $\mathcal{I}(M_{t_i})\to 1$. }
\end{center}
As a corollary of the  Alexandrov-Fenchel inequalities \cite{GuanLiQuermass} we get that
\begin{center}
\emph{Any mean convex, star-shaped solution of the VPMCF with uniformly bounded curvature exists for all times and converges to a round sphere.}
\end{center}
In Section~\ref{sec:Alex}, we remind the reader of the notion of an \emph{Alexandrov immersion}, which is -- roughly speaking -- an $n$-dimensional immersion that bounds an $(n+1)$-dimensional manifold that is immersed in $\mathbb{R}^{n+1}$ (see Definition~\ref{Def:Alexim} for the precise definition). Alexandrov immersions bound an $(n+1)$-dimensional ``domain'' that is allowed to have a certain kind of self-overlaps. The definition goes back to Alexandrov~\cite{Alexandrov} in his work about closed surfaces of constant mean curvature in Euclidean space\footnote{Note that being Alexandrov immersed is not called like this in \cite{Alexandrov}, of course. Also, the property that the $(n+1)$-dimensional domain is \emph{immersed} in $\mathbb{R}^{n+1}$ is missing in this paper (there, only the expression \emph{smooth mapping} is used). But it is meant to be part of the definition as it is used in the proof.}. It has also been been used successfully for example by Brendle in his work about minimal tori in $\mathbb{S}^3$, see \cite{BrendleAlex}. Using maximum principle arguments in a ``one-sided situation'' we prove the following result:
\begin{center}
\emph{The property of being Alexandrov immersed is preserved under the VPMCF as long as $\ov{H}\geq 0$ (using the outer unit normal for the definition of $\ov{H}$). }
\end{center}
Unfortunately, we found an (even embedded) example where $\ov{H}\geq 0$ is lost along the VPMCF. We explain this example in Appendix~\ref{appendix1}.

\section*{Acknowledgments}

We would like to thank Karsten Gro\ss e-Brauckmann for discussions about Alexandrov immersions and for his interest in our work. The second author is supported by the DFG (MA 7559/1-2) and thanks the DFG for the support.

\section{Notation}\label{sec:Notation}
We complement the definitions from the introduction with the formula of the signed volume on the immersion $M_t$
%
%
%
 \[V(t):= \on{Vol}(M_t) = \frac 1 {n+1}\int_{M_t} \ip{X}{\nu} d\mu.\]
The defining property of VPMCF is that it preserves this volume, that is, 
\[\on{Vol}^{n+1}(M_t)=\on{Vol}^{n+1}(M_0)=:V_0\ .\]
In particular if $\Omega_t\subset \bb{R}^{n+1}$ is a domain with smooth boundary $\partial \Omega_t=M_t$, then the divergence theorem implies that $\on{Vol}^{n+1}(M_t)=\mathcal{L}^{n+1}(\Omega_t)$ where $\mathcal{L}^{n+1}$ is the $(n+1)$--dimensional Lebesgue measure.

We will regularly use the \emph{extrinsic} diameter of the immersion $M_t$ defined by
\begin{equation}
\operatorname{diam}(M_t):=\operatorname{diam}_{\on{ext}}(M_t) = \max_{(x,y)\in M^n\times M^n}|X(x,t) - X(y,t)|\ .\label{eq:diamdef}
\end{equation}
This is distinct to the intrinsic metric diameter of $M_t$, $\on{diam}_{\on{int}}(M_t)$ where we have the inequality $\diam(M_t)\leq \on{diam}_{\on{int}}(M_t)$.

\section{Diameter bounds and $L^2$ bounds of the Lagrange parameter}\label{sec:diamL2}
\begin{lemma} \label{lemma:diffdiam}
 Let $X:M^n\times [0,T) \to\R^{n+1}$ be a closed VPMCF for $t\in[0,T)$. Then the extrinsic diameter (as in equation \eqref{eq:diamdef}) satisfies
 \begin{align*}
  \frac{d}{dt} \diam (M_t) = \langle -(H(x,t) -\ov{H}(t)) \nu(x,t) + (H(y,t) - \ov{ H }(t)) \nu(y,t), \tfrac{X(x,t) - X(y,t)}{|X(x,t) - X(y,t)|}\rangle
 \end{align*}
for almost every $t\in [0,T)$, where $(x,y)\in M^n\times M^n$ are points where the maximum of the diameter is attained.
\end{lemma}
\begin{proof}
 The Lipschitz continuous map $t\mapsto \operatorname{diam}(M_t)$ can be differentiated for almost every $t$ using Hamilton's trick (see e.g. \cite[Lemma 2.1.3]{Mantegazza}), which implies
 \begin{align*}
   \frac{d}{dt} \diam (M_t) & = \frac{1}{|X(x,t) - X(y,t)|} \langle \partial_t X(x,t) - \partial_t X(y,t), X(x,t) - X(y,t) \rangle\\
   &= \langle -(H(x) -\ov{H}) \nu(x) + (H(y) - \bar H) \nu(y), \tfrac{X(x) - X(y)}{|X(x) - X(y)|}\rangle,
 \end{align*}
where $(x,y)\in M^n\times M^n$ are points where the maximum of the diameter is attained. 
\end{proof}

 \begin{lemma} \label{lemma:derivativeDiam}
   Let $X:M^n\times [0,T) \to\R^{n+1}$ be a closed VPMCF for $t\in[0,T)$. Then the extrinsic diameter satisfies
   \begin{align*}
    \frac{d}{dt} \diam (M_t) & \leq \ov{H} \left[\langle \nu^x,  \tfrac{X(x,t) - X(y,t)}{|X(x,t) - X(y,t)|}\rangle - \langle \nu^y, \tfrac{X(x,t) - X(y,t)}{|X(x,t) - X(y,t)|}\rangle\right] - \frac{4n}{\diam(M_t)}\\
    &\leq 2 |\ov{H}| - \frac{4n}{\diam(M_t)}
   \end{align*}
for almost every $t\in [0,T)$.
 \end{lemma}
 
 \begin{proof}
  We consider the squared distance function $\varphi: M^n \times M^n \times [0,T) \ra \bb{R}$, $\varphi(x,y,t): = |X(x,t) - X(y,t)|^2$. We write $(z_1, \ldots, z_{2n}) := (x_1, \ldots, x_n ,y_1, \ldots, y_n)$ if needed. We leave out the time dependence in the notation from now on. We calculate the second derivatives of $\varphi$. But we only want to use them for $x$ and $y$ such that the maximum in $\diam(M_t)^2= \max_{(x,y)\in M^n\times M^n}|X(x,t) - X(y,t)|^2$ is attained. Thus, we compute
\[\partial_{x_i}\varphi = 2\ip{\partial_{x_i} X(x)}{X(x)-X(y)}, \qquad \partial_{y_i}\varphi = -2\ip{\partial_{y_i}X(y)}{X(x)-X(y)}\]
\begin{flalign*}
\n^2_{x^i x^j}\varphi &= -2h_{ij}(x) \ip{\nu(x)}{X(x)-X(y)}+2g_{ij}(x)\\
\n^2_{x^i y^j}\varphi &=-2\ip{\partial_{x_i} X (x)}{\partial_{y_j}X(y)}\\
\n^2_{y^i y^j}\varphi &= 2h_{ij}(y) \ip{\nu(y)}{X(x)-X(y)}+2g_{ij}(y)
\end{flalign*}
At almost every time $t$, $\diam(M_t)^2$ is differentiable, and by Hamilton's trick, we know that
\[\ddt{} \diam(M_t)^2 = 2 \langle -(H(x) -\ov{H}) \nu(x) + (H(y) - \bar H) \nu(y), {X(x) - X(y)}\rangle\]
for $x$ and $y$ attaining the maximum. At such points we have that $0=\partial_{x_i}\varphi = \partial_{y_i}\varphi$ and  $0\geq \n^2_{z^iz^j} \varphi$.  This particularly implies that $X(x) -X(y)$ is normal to the surface in $x$ and $y$, so the tangent spaces at $x$ and $y$ are parallel and we may take local orthonormal coordinates so that $\partial_{x_i} X(x)=\partial_{y_i}X(y)=e_i$. Furthermore, 
\[0\geq \sum_i \n^2_{e_i^x-e_i^y,e_i^x-e_i^y}\varphi = -2 \ip{\nu(x)}{X(x)-X(y)}H(x)+2\ip{\nu(y)}{X(x)-X(y)}H(y) +8n\]
which implies
\[\ddt{}\diam^2(M_t)\leq 2\ov{H} [\ip{\nu(x)}{X(x)-X(y)}-\ip{\nu(y)}{X(x)-X(y)}]-8n\]
or alternatively
\begin{align}\begin{split}\label{eq:ineqdt}
 \ddt{}\diam(M_t)&\leq \ov{H} \left[\ip{\nu(x)}{\frac{X(x)-X(y)}{|X(x)-X(y)|}}-\ip{\nu(y)}{\frac{X(x)-X(y)}{|X(x)-X(y)|}}\right]-\frac{4n}{\diam(M_t)}\\
&=:2\left(\ov{H} \sigma-\frac{2n}{\diam (M_t)}\right)\\
&\leq 2\left(|\ov{H}|-\frac{2n}{\diam (M_t)}\right) \end{split}
\end{align}
as stated in the lemma. Here, $\sigma \in \{1,0,-1\}$ depends on the configuration of normals and positions. 
 \end{proof}
 
 \begin{remark}
  Lemma~\ref{lemma:derivativeDiam} implies that the diameter along the VPMCF is bounded on every bounded time interval $[0,T)$, $T<\infty$, if one is able to show a bound $\sup_{t\in[0,T)}|\ov{H}(t)| \leq c$. Unfortunately, such a bound is hard to get in a very general setting. See Proposition \ref{prop:uniformdiambound} for a weakening of this condition.
 \end{remark}

\begin{remark}
The above calculation is to some extent optimal in bounding diameter, even in the case of convex curves. Indeed, even in the convex setting, the nonlocal term can force the diameter to increase initially, before converging back to a circle. See example \ref{ex:convexcurve} for details.
\end{remark}

  \begin{defses}\label{Def:Alexim}
An immersion $X:M^n\ra \bb{R}^{n+1}$ is \emph{Alexandrov immersed} if there exists an $(n+1)$-dimensional manifold $\ov{\Omega}$ with $\partial \ov{\Omega} = M^n$ and an immersion $G:\ov{\Omega} \ra \bb{R}^{n+1}$ such that $G|_{\partial \ov{\Omega}}$ also parametrises $\on{Im}(X)$. 
\end{defses}

\begin{remark}
 For an Alexandrov immersion, there is obviously a natural notion of an an \emph{inner unit normal $\tilde\nu$} (the one where $G^\ast(\tilde\nu)$ shows into $\Omega$) and an \emph{outer unit normal $-\tilde \nu$}.
\end{remark}

 \begin{cor}\label{cor:Alexdiam}
  Let $X:M^n\times [0,T) \to\R^{n+1}$ be a closed VPMCF for $t\in[0,T)$ that is an Alexandrov immersion for all $t\in [0,T)$. Then we choose $\nu$ to be the outer unit normal. In this case 
  \begin{align*}
  \ip{\nu(x)}{\frac{X(x)-X(y)}{|X(x)-X(y)|}}-\ip{\nu(y)}{\frac{X(x)-X(y)}{|X(x)-X(y)|}} = 2
  \end{align*}
 for $x(t)$ and $y(t)$ such that the maximum in the definition of the diameter is attained. As a consequence, we have that
 \begin{align*}
   \ddt{}\diam(M_t) \leq  2\left(\ov{H} - \frac{2n}{\diam(M_t)}\right)
 \end{align*}
 for almost every $t\in [0,T)$.
 \end{cor}
 
 \begin{proof}
  Since there are no points $X(\tilde x)$ that are further apart from $X(y)$ in $\R^{n+1}$ than $X(x)$, the outer unit normal $\nu(x)$ must agree with $\frac{X(x)-X(y)}{|X(x)-X(y)|}$. For $X(y)$, an analogous argument works. 
 \end{proof}

 \begin{cor} \label{cor:Alexhleq0}
   Let $X:M^n\times [0,T) \to\R^{n+1}$ be a closed VPMCF that is Alexandrov-immersed for all $t\in [0,T_{\text{max}})$ and $\ov{H}\leq 0$ on $t\in [0,T_{\text{max}})$, where $\ov{H}$ is computed with respect to the outer unit normal.  Then $T_{\text{max}}\leq \frac{\diam(M_0)^2}{8n}$.
 \end{cor}
 
 \begin{proof}
  The assumption $\ov{H}\leq 0$ and Corollary~\ref{cor:Alexdiam} imply
\begin{align*}
 \ddt{}\diam(M_t) \leq -\frac{4n}{\diam(M_t)}, \ \ \ \ t\in [0,T_{\text{max}})
\end{align*}
which is equivalent to 
\begin{align*}
 \diam(M_t)^2 \leq \diam(M_0)^2 - 8nt \ \ \text{for } t\in [0,T_{\text{max}})
\end{align*}
via integration. So we have that
\begin{align*}
\limsup_{t \to T_{\text{max}}} \diam(M_t)^2 \leq \diam(M_0)^2 - 8nT_{\text{max}}
\end{align*}
which leads to a contradiction for $T_{\text{max}}> \frac{\diam(M_0)^2}{8n}$.
 \end{proof}
 
 \begin{remark}
  This above statement should be interpreted in the following way: If we know the flow exists for quite some time \emph{and stays Alexandrov immersed}, then $\ov{H} \leq 0$ is not preserved. For more on this see also Theorem \ref{thm:asymptotics}.
 \end{remark}

\begin{lemma}
Let $X:M^n\times [0,T) \to\R^{n+1}$ be a VPMCF with enclosed volume $V_0 \not =0$. Then we can express $\ov{H}$ as
\begin{align}\label{eq:barH}
 \ov{H} = \frac{\int_M\ip{\partial_t X}{X-x}d\mu+n|M_t|}{(n+1)V_0}
\end{align}
for any $x=x(t) \in \R^{n+1}$.
\end{lemma}
\begin{proof}
 We have that $\Delta |X-x|^2= -2H\ip{\nu}{X-x}+2n$ and so by the divergence theorem
\begin{equation}
|M_t| = \frac 1 n\int_M H\ip{\nu}{X-x}d\mu\ .\label{eq:HsuppisArea}
\end{equation}
The formula for the signed volume enclosed by $M_t$ reads $V(t) =\frac{1}{n+1}\int_{M_t}\ip{X-x}{\nu}d\mu$ for any $x=x(t)\in \R^{n+1}$.
We compute
\[\int_M\ip{X_t}{X-x}d\mu = \ov{H}\int_M \ip{\nu}{X-x}d\mu -\int_M H\ip{\nu}{X-x}d\mu\]
and so
\[\ov{H} = \frac{\int_M\ip{\partial_t X}{X-x}d\mu+\int_M H\ip{\nu}{X-x}d\mu}{\int_M \ip{\nu}{X-x}d\mu}=\frac{\int_M\ip{\partial_t X}{X-x}d\mu+n|M_t|}{(n+1)V_0}\]
using the volume preserving property. 
\end{proof}

\begin{proposition}\label{prop:ElenasArgument2}
Suppose that $(M_t)_{t\in[0,T)}$ satisfies  $V_0\neq 0$ and $\text{diam}(M_t)<R (t)$ with $R(s) \leq R(t)$ for $s<t\leq T$. Then we have for $t\in [0,T)$,
\[\int_0^t \ov{H}^2(\tau) d\tau \leq \frac{ 2|M_0|^2}{V_0^2(n+1)^2}(R(t)^2+n^2t) \]
\end{proposition}
\begin{proof}
First we choose $x(t)\in \R^{n+1}$ to be any point such that $M_t\subset B_R(x(t))$.
We have that 
\[\ddt{}|M_t| = -\int_M H(H-\ov{H})d\mu=-\int_M (H-\ov{H})^2d\mu = -\int_M|\partial_t X|^2d\mu.\]
Formula (\ref{eq:barH}) implies
 \begin{align*}
  |\ov{H}| \leq \frac{1}{(n+1)|V_0|}\left(\sqrt{\int_M|\partial_t X|^2d\mu}\sqrt{\int_M|X-x|^2d\mu}+n|M_0|\right).
 \end{align*}
and thus
\[\ov{H}^2\leq \frac 2 {(n+1)^2V_0^2}\left(-\ddt{}|M_t|\int_M |X-x|^2d\mu +n^2|M_0|^2\right)\ .\]
Estimating $|X-x|^2\big|_{s}\leq \text{diam}^2M_s \leq R^2(s) \leq R^2(t)$ for $s<t$ and using $|M_s|\leq |M_0|$ again,  we know now
\begin{align*}
 \int_0^t\ov{H}^2ds \leq \frac{2}{V_0^2(n+1)^2}|M_0|^2(R^2(t)+ n^2t)\ .
\end{align*}

\end{proof}
 
 \begin{proposition}\label{prop:dtdiam}
  Let $X:M^n\times [0,T) \to\R^{n+1}$ be a closed VPMCF for $t\in[0,T)$ with enclosed volume $V_0\not =0$. Then there are constants $c,C$ only depending on $M_0$ such that 
  \begin{align*}
   \diam(M_t) \leq C(1+t) e^{c\sqrt{t}}
  \end{align*}
  for $t\in [0,T)$

 \end{proposition}
 
 \begin{proof}
We have 
\[\ddt{}\diam(M_t)\leq 2|\ov{H}|-\frac{4 n}{\diam (M_t)}\leq 2|\ov{H}|\]
and from (\ref{eq:barH})
\begin{align*}
 |\ov{H}| & \leq \frac{1}{(n+1)|V_0|}\left(\sqrt{-\ddt{}|M_t|}\sqrt{|M_0|} \diam(M_t)+n|M_0|\right)\\
 &: = c_1 + c_2 \diam(M_t)\sqrt{-\ddt{}|M_t|}.
\end{align*}
We pick 
\begin{align}\label{eq:psi}
 \log \psi = -2c_2\int_0^t \sqrt{-\ddt{}|M_t|}\geq -2c_2 \sqrt{t} (|M_0|-|M_t|)^{\frac{1}{2}} \geq -2c_2 \sqrt{t} |M_0|^{\frac{1}{2}}
\end{align}
so that
\[\frac{1}{\psi}\ddt{}(\diam(M_t) \psi) = \ddt{}\diam(M_t)-2c_2 \diam(M_t)\sqrt{-\ddt{}|M_t|}\leq 2c_1\]
and thus (using that $\psi\leq 1$ and (\ref{eq:psi}))
\begin{align*}
 \diam(M_t) &\leq \frac{1}{\psi(t)}\diam(M_0)+2c_1\frac{1}{\psi(t)}\int_0^t\psi(\tau) d\tau\\
 &\leq \left(\diam(M_0) + 2c_1 t\right) \frac{1}{\psi(t)} \\
 &\leq C(1+t) e^{2c_2|M_0|^{\frac{1}{2}}\sqrt{t}}.
\end{align*}

\end{proof}
 
\begin{cor}\label{cor:L2Hbound}
  Let $X:M^n\times [0,T) \to\R^{n+1}$ be a closed VPMCF for $t\in[0,T)$ with enclosed volume $V_0\not =0$. Then there are a constants $c,C>0$ only depending on $n$, $V_0$ and $M_0$ such that 
\[\int_0^t \ov{H}^2(s) ds \leq C (1+t^2) e^{c\sqrt{t}}. \]
\end{cor}
\begin{proof}
 Put together Proposition~\ref{prop:ElenasArgument2} and \ref{prop:dtdiam}.
\end{proof}
\begin{remark}
\begin{enumerate}
 \item  In general, we can not expect to have a bound of the form $\int_0^t \ov{H}^2(s) ds \leq C$, where $C$ is independent of $t$ because round spheres are (stationary) solutions of the VPMCF, and they satisfy $\int_0^t \ov{H}^2(s) ds = c^2 t$. We do not know whether the $e^{\sqrt{t}}$-growth can be improved (e.g.\ to a linear growth in $t$) in this very general setting.
 \item Integrability conditions of a Lagrange multiplier was successfully used earlier, see for example the work of Rupp~\cite{Rupp} on the volume preserving Willmore flow. But there, the fixed-enclosed-volume constraint is of lower order compared to the order of the flow, therefore Rupp was able to get strong results with the integrability.
\end{enumerate}

\end{remark}

\subsection{Optimality of diameter bounds}
In this subsection we consider a curve $\gamma:\mathbb{S}^1 \to \mathbb{R}^2$, where $\partial_s$ is differentiation with respect to arclength and $\nu = -J \partial_s \gamma$ is the rotation of the tangent $\partial_s\gamma$ by $\frac{\pi}{2}$ in negative direction. In order to be consistent with the definitions for surfaces, we use the definition $\kappa = -\langle \partial_s^2 \gamma, \nu\rangle$.  The \emph{area preserving curve shortening flow} reads
\begin{align*}
 \partial_t \gamma = -(\kappa -\ov{\kappa})\nu
\end{align*}
for $\ov{\kappa} : = \frac{\int\kappa ds}{L(\gamma_t)}$. Here, $ds = |\partial_x \gamma|dx$ denotes integration with respect to arclength and $L(\gamma_t)$ is the length of $\gamma_t: =\gamma(\cdot,t)$ for the family $\gamma:\mathbb{S}^1\times[0,T) \to \R^2$. Of course, as $\gamma$ is closed and  $\int\kappa ds$ is the turn of the tangents, we have that $\ov{\kappa} =  \frac{2\pi m}{L(\gamma_t)}$. If $\gamma_0$ is embedded, then $m=1$.\\

On a general curve we have that 
\begin{align}\label{cond1}
 \kappa_0(x) = \kappa_0(y), \ \ \frac{\gamma_0(x) -\gamma_0(y)}{|\gamma_0(x) -\gamma_0(y)|} = \nu_0(x), \ \ \frac{\gamma_0(x) -\gamma_0(y)}{|\gamma_0(x) -\gamma_0(y)|} = - \nu_0(y),
\end{align}
where $x, y \in \mathbb{S}^1$ are such that the maximum in $\diam (\gamma_0)$ is attained. Note that for a positively oriented  curve $\gamma_0$, $\nu_0$ is the outer unit normal. The proof of Lemma~\ref{lemma:derivativeDiam} shows that 
\begin{align*}
  0&\geq -2 \langle \nu_0(x), \gamma_0(x) - \gamma_0(y) \rangle \kappa_0(x) + 2 \langle \nu_0(y),  \gamma_0(x) - \gamma_0(y) \rangle \kappa_0(y) + 8.
\end{align*}
Then (\ref{cond1}) implies that
\begin{align*}
 \langle \nu_0(x), \gamma_0(x) - \gamma_0(y) \rangle &= |\gamma_0(x) -\gamma_0(y)| = \diam(\gamma_0)\\
  \langle \nu_0(y), \gamma_0(x) - \gamma_0(y) \rangle& = -|\gamma_0(x) -\gamma_0(y)| = -\diam(\gamma_0)
\end{align*}
and thus
\begin{align}
 \kappa_0(x) \geq \frac{2}{\diam(\gamma_0)}.\label{eq:diamineq}
\end{align}
This inequality can be interpreted geometrically: The curvature at the points where the maximum is attained cannot be smaller than $\frac{2}{\diam(\gamma_0)}$ -- otherwise the straight line realizing the maximal distance would be attained at other points (tilting that line would increase the distance). Bearing this is mind, we construct a somehow sharp example. 

We now provide an example demonstrating, somewhat counter-intuitively, that even on a convex curve, (which we know will converge to a round circle in infinite time \cite{HuiskenConvexVPMCF}), the diameter will initially increase while the curvature at extremal points initially decreases.

\begin{figure}[h!]
 \begin{tikzpicture}
    \draw (0,4) arc (80:100:2cm);
        \draw (0,0) arc (-80:-100:2cm);

\draw[->] (-1, -0.04) to (2, -0.04);
\draw[->] (-0.34, -0.2) to (-0.34, 4.2);
\draw[out=-10,in=90] (0,4) to (0.2, 3.5);
\draw[out=-90,in=90] (0.2,3.5) to (0.2, 0.5);
\draw[out=-90,in=10] (0.2,0.5) to (0, 0);

\draw[out=190,in=90] (-0.69,4) to (-0.89, 3.5);
\draw[out=-90,in=90] (-0.89,3.5) to (-0.89, 0.5);
\draw[out=-90,in=170] (-0.89,0.5) to (-0.69, 0);

\coordinate[label = right: ${\gamma_0}$] (C) at (0.2,2);
 \coordinate (d) at (-0.34,-0.035);
	  \fill (d) circle (1pt);
	  \coordinate[label= below:${\gamma_0(y)}$] (g) at (-0.34,-0.1);
	   \coordinate (e) at (-0.34,4.035);
	  \fill (e) circle (1pt);
	     \coordinate[label= above:${\gamma_0(x)}$] (h) at (-0.34,4.1);
	   \coordinate (f) at (-0.34,2);
	  \fill (f) circle (1pt);
 \end{tikzpicture}
 \caption{An embedded, convex initial curve for which the diameter increases.}\label{Figdiam}
\end{figure}
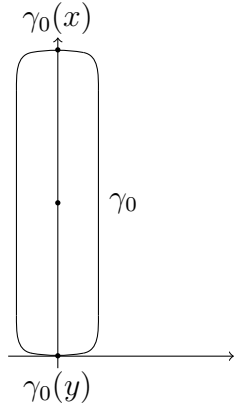
\begin{example}\label{ex:convexcurve} In this example we show that $ \frac{d}{dt} \diam (\gamma_t)\big|_{t=0}>0 $ can occur, even on closed convex curves. We construct an explicit example as in Figure \ref{Figdiam} as follows: First we set $\diam(\gamma_0) = 1$. We put a circle with radius $r: =\frac{1}{2}$ at the midpoint of the line realizing the maximal distance  (w.l.o.g. $\{0\}\times[0,1]\subset\R^2$). So we use very short arcs of the circle at the points $(0,1) = \gamma_0(x)$ and $(0,0) = \gamma_0(y)$. Thus, in $x$ and $y$, the curvature is precisely equal to $2$. We close this curve (smoothly) by straight lines (going from south to north and the other way round at the end of the arcs). The diameter is clearly realised by the points $x$ and $y$.
\end{example}

In Example \ref{ex:convexcurve} we have equality in equation \eqref{eq:diamineq}. With Lemma~\ref{lemma:diffdiam}, we know that
\begin{align*}
  \frac{d}{dt} \diam (\gamma_t)\big|_{t=0} & = \langle -2(\kappa_0(x) - \ov{\kappa}_0)\nu_0(x), \frac{\gamma_0(x) - \gamma_0(y)}{|\gamma_0(x) - \gamma_0(y)|} \rangle 
   = -4 (1- \frac{\pi}{L(\gamma_0)})
\end{align*}
taking into account our situation. The length of $\gamma_0$ can be chosen to be $L(\gamma_0) = 2+\epsilon$. The curve only needs to satisfy $L(\gamma_0)<\pi$ to get an example where $ \frac{d}{dt} \diam (\gamma_t)\big|_{t=0} >0$, in particular $\kappa_0(x) - \ov{\kappa}(0) <0$.

There is something else worth noticing about the above example: The evolution equation of the curvature along the area preserving curve shortening flow is 
   \begin{align*}
    \partial_t \kappa = \partial_s^2 \kappa + (\kappa - \ov{\kappa}) \kappa^2.
   \end{align*}
We use this for our initial curve $\gamma_0$ at the points $x$ and $y$ (where the maximal distance is attained). Since $\kappa \equiv 2$ on a small neighborhood around $x$ and $y$, we know $\partial_s^2 \kappa(x) = 0 = \partial_s^2 \kappa(y)$ and hence 
\begin{align*}
 \partial_t \kappa(x) \big|_{t=0} = \partial_t \kappa(y) \big|_{t=0} = (\kappa_0(x) - \ov{\kappa}(0)) \kappa^2(x)<0.
\end{align*}
This means that, at the points $\gamma_0(x)$ and $\gamma_0(y)$, the curve becomes even a bit flatter, and the diameter grows a bit. This is surprising, because eventually, this curve converges to a round circle. This example shows that the short-time behavior of the flow can be somehow counterintuitive.

\section{Finite time singularities and monotonicity results}\label{sec:4}

\subsection{Monotonicity formula}

We derive a monotonicity formula for VPMCF and demonstrate various corollaries following a similar course to \cite{Eckerbook}. This has been calculated by the second author in \cite{ElenaDiss} already, inspired by \cite{Atha03}. Note that also \cite{TakasaoI} contains a version of a monotonicity formula.

The evolution equation of the area measure is
\[\ddt{}d \mu =-(H-\ov{H})H d \mu\ .\]
We now set $\Phi = \Phi_{(x_0,t_0)}$ the standard ``ambient heat kernel'',  
\[\Phi_{(x_0,t_0)} = \frac{1}{(4\pi(t_0-t))^\frac n2}e^{-\frac{|x-x_0|^2}{4(t_0-t)}}, \ \ x\in \R^{n+1}, t<t_0, \]
where in the following calculation we will suppress the subscript for simplicity of notation. 

\begin{proposition}[Monotonicity formula]\label{prop:Monoton1}
Suppose that $M_t$ satisfies smooth VPMCF. We define
\[\psi(t):=e^{-\frac 12\int_0^t \ov{H}^2(s) ds}\]
and let $\phi$ be a $C^2$-function that is either defined on $M^n \times [0,T)$ or on a neighborhood of $M_t$ in $\R^{n+1}$. Then we have that
\begin{flalign*}
\ddt{} \left( \psi\int_{M_t} \phi\Phi d \mu \right)&=  \psi\int_{M_t} \Phi\ho \phi d \mu +\psi\int_{\partial M_t} \Phi \n_\mu \phi -\phi \n_\mu \Phi d \mu^\partial\\
&\qquad -\frac{\psi}{2}\int_{M_t}\left[\left|(H-\ov{H})-\frac{\ip{\nu}{x-x_0}}{2(t_0-t)}\right|^2+\left| H-\frac{\ip{\nu}{x-x_0}}{2(t_0-t)}\right|^2 \right]\phi\Phi d \mu\ .
\end{flalign*}

\end{proposition}

\begin{proof} The proof can be found in \cite{ElenaDiss}, see also \cite[Remark before Proposition~4.10]{Elena15}. For the convenience of the reader we repeat it here.
For general functions $\phi$ and $\Phi:\bb{R}^{n+1}\times[0,T) \ra \bb{R}$ we have that
\begin{flalign*}
\ddt{}\int_{M_t} \phi\Phi d \mu 
&= \int_{M_t} \Phi\ho \phi +\div(\Phi \n \phi -\phi \n \Phi) +\phi\left(\ddt{\Phi}+\Delta \Phi\right)-H(H-\ov{H})\phi\Phi d \mu\\
&= \int_{M_t} \Phi\ho \phi +\div(\Phi \n \phi -\phi \n \Phi)d \mu\\
&\qquad +\int_{M_t}\left[\frac 1 \Phi\left(\left(\pard{}{t}+\ov{\Delta}\right) \Phi-\ov\n^2_{\nu\nu} \Phi -(H-\ov{H})\ov\n_\nu \Phi-H\ov\n_\nu \Phi\right)-H(H-\ov{H})\right]\phi\Phi d \mu.
\end{flalign*}
For our choice of $\Phi$, we get that
\[\pard{\Phi}{t} = \left(\frac{n}{2(t_0-t)}-\frac{|x-x_0|^2}{4(t_0-t)^2}\right)\Phi,\qquad \ov \n_V \Phi = -\frac{\ip{V}{x-x_0}}{2(t_0-t)}\Phi, \]\[\ov \n^2_{VW} \Phi = \left(\frac{\ip{V}{x-x_0}\ip{W}{x-x_0}}{4(t_0-t)^2}-\frac{\ip{V}{W}}{2(t_0-t)}\right)\Phi\ .\]
In particular we have
\[\left(\pard{}{t}+\ov{\Delta}\right) \Phi=\frac{-1}{2(t_0-t)}\Phi, \qquad \ov\n^2_{\nu\nu} \Phi = \left(\left(\frac{\ip{\nu}{x-x_0}}{2(t_0-t)}\right)^2-\frac{1}{2(t_0-t)}\right)\Phi\]
and so the contents of the square brackets in the last integrand above are
\begin{flalign*}-&\left(\frac{\ip{\nu}{x-x_0}}{2(t_0-t)}\right)^2+(H-\ov{H})\frac{\ip{\nu}{x-x_0}}{2(t_0-t)}+H\frac{\ip{\nu}{x-x_0}}{2(t_0-t)}-H(H-\ov{H})\\
&=-\frac 1 2\left|(H-\ov{H})-\frac{\ip{\nu}{x-x_0}}{2(t_0-t)}\right|^2-\frac 1 2\left| H-\frac{\ip{\nu}{x-x_0}}{2(t_0-t)}\right|^2+\frac 1 2\ov{H}^2\ .
\end{flalign*}
We have chosen $\psi(t)=e^{-\frac 12\int_0^t \ov{H}^2(s) ds}$ so that $\ddt \psi= -\frac 1 2 \ov{H}^2$ which yields
\begin{flalign*}
\ddt{}\int_{M_t} \psi\phi\Phi d \mu &= \int_{M_t}\psi \Phi\ho \phi d \mu +\psi\int_{\partial M_t} \Phi \n_\mu \phi -\phi \n_\mu \Phi d \mu^\partial\\
&\qquad -\frac{\psi}{2}\int_{M_t}\left[\left|(H-\ov{H})-\frac{\ip{\nu}{x-x_0}}{2(t_0-t)}\right|^2+\left| H-\frac{\ip{\nu}{x-x_0}}{2(t_0-t)}\right|^2 \right]\phi\Phi d \mu\ .
\end{flalign*}
\end{proof}

\begin{remark}
We have that\footnote{Is this remark worth keeping? It is sort of tangentially interesting but it doesn't really go anywhere...} 
\[\ddt{}\log |M_t| = -\fint_{M_t} H^2 d \mu + \ov{H}^2\]
and so we may rewrite
\[\psi(t) = \sqrt{\frac{|M_0|}{|M_t|}}e^{\int_0^t\fint_{M_s} H^2 d \mu ds}\]
\end{remark}
\begin{remark}\label{scaleh1}
Consider a parabolic rescaling of $M_t$, $\widetilde{M}_s$ where $t=\l^2 s+t_0$, and ${M}_t = \l \widetilde{M}_s +x_0$. Then $\ov{H}(\widetilde{M}_s)=\l\ov{H}(M_t)$ and we have that
\[\log \widetilde{\psi}(s) = \int_0^s\ov{H}(\widetilde{M}_{\hat{s}})^2d\hat{s}=\int_0^s\l^{2}\ov{H}({M}_{t(\hat s)})^{2}d\hat{s}=\int_{t_0}^{t}\ov{H}({M}_{\hat{t}})^{2}d\hat{t}=\log\psi(t)\ ,\]
which implies that our object to differentiate is invariant under parabolic rescaling. 
\end{remark}
\subsection{VPMCF blowups are ancient MCFs}
We will use the following MCF terminology.
\begin{defses}
Given any $X(x,t)$ satisfying VPMCF the parabolic rescaling about a point $p\in \bb{R}^{n+1}$ at time $T$ given by \[X^\lambda_p(x,\tau) := \lambda (X(x, \lambda^{-2} \tau+T)-p)\] for $\tau\leq 0$. Equivalently we may write this as 
\[M^{\lambda, p}_\tau =\lambda(M_{\lambda^{-2}\tau+T}-p)\ .\] 
A parabolic rescaling of VPMCF gives another solution of VPMCF (although with a different fixed volume).

\end{defses}

\begin{defses}
Let $M_t$ be a mean curvature flow on the time interval $[0,T)$, and suppose that the second fundamental form blows up at time $T$ that is, $\limsup_{t\ra T} \sup_{M_t}|A| = \infty$. Then we say that this is a singularity of VPMCF where
\begin{enumerate}[label=\roman*)]
\item the singularity is of type I if there exists a $C>0$ such that
\[\sup_{M_t}|A|^2 \leq \frac C {T-t}\]
for all $t<T$ and
\item the singularity is of type II otherwise.
\end{enumerate}
\end{defses}
\begin{defses}
Let $M_t$ be a mean curvature flow on the time interval $[0,T)$. Then
\begin{enumerate}[label=\roman*)]
\item Suppose we are given a point $p \in \bb{R}^{n+1}$ and a sequence $\lambda_i\ra \infty$. Then a Type I (or centred) blow up sequence is the sequence of solutions to VPMCF given by
\[M^i_\tau = \lambda_i(M_{\lambda_i^{-2}\tau + T}-p)\]
for $\tau \in [-\lambda_i^2 T,0)$.
\item Assuming that the second fundamental form blows up as $t\ra T$, then we follow Hamilton~\cite{Hamilton} and define a Type~II blow up sequence by a sequence times $t_i$ and points $x_i\in M_{t_i}$ such that
\[\lambda_i:=|A(x_i,t_i)|^2(T-i^{-1}-t_i)=\max_{(x,t)\in M^n \times[0,T-i^{-1}]}|A(x,t)|^2(T-i^{-1}-t)\]
Then the Type II blow up along $(x_i,t_i)$ is given by
\[M^i_\tau = \lambda_i(M_{t_i+\lambda_i^{-2}\tau}-x_i)\]
for $\tau\in[-\lambda_i^2 t_i, \lambda_i^2(T-i^{-1}-t_i))$
\end{enumerate}
\end{defses}
Once we have uniformly bounded curvature, $\ov{H}$ is bounded and Shi-type estimates hold to all orders, as in \cite[Proposition 3.22]{Eckerbook}. As a result, if a type I singularity occurs, then any type I blowup sequence will converge locally smoothly to a smooth flow for $\tau\in(-\infty, 0)$ (which is possibly empty depending on the $p$). This will be referred to as an ancient flow. Later in Lemma~\ref{lem:homothetic} we will see that this is a self similar solution of MCF. If we have a type II  singularity then the Type II blowup sequence converges smoothly to a flow for $\tau\in(-\infty, \infty)$ \cite{Hamilton}. 

We are now in a position to prove the following:
\begin{theorem}\label{thm:asymtflow}
Suppose that a singularity of VPMCF occurs at a finite time. Then 
\begin{enumerate}[label=\alph*)]
\item If the singularity is type I then any locally smoothly converging subsequence of the Type I blow up sequence which converges smoothly is an ancient solution to MCF. 
\item If the singularity if type II then any locally smoothly converging subsequence of a type II blow up sequence yields an eternal solution to MCF.
\end{enumerate}
\end{theorem}
\begin{proof}
We first consider the type I blow up case: By dilating we may assume wlog that  $|A|^2(x,t)<\frac 1 {T-t}$. By Corollary \ref{cor:L2Hbound}, we know that for any finite $T$
\[\int_0^T \ov{H}^2(t) dt<C\ .\]
By continuous dependence of an integral on it's domain of definition (Lebesgue's differentiation theorem) we have that there exist $T_i\ra T$ such that
\begin{equation}\int^T_{T_j} \ov{H}^2 dt = j^{-1}\label{eq:Hjest}
\end{equation}
and therefore, writing $t = \lambda_i^{-2}(\tau+T)$ then
\[j^{-1}=\int^T_{T_j} \ov{H}^2(t) dt =\int^0_{-\lambda^2(T-T_j)}\l^{-2}\ov{H}^2(\tau)d\tau=\int^0_{-\lambda_i^2(T-T_j)}(\ov{H}^i)^2(\tau)d\tau \ .\]
where $\ov{H}^i$ is the average mean curvature of $M_\tau^i$ of the type I blowup. Due to the type I singularity hypothesis, the second fundamental form of $M^i_\tau$ satisfies 
\[|A^i(x,\tau)|^2\leq \frac{1}{|\tau|}\ .\]
Furthermore, we have that
\begin{equation}
\left|\ddt{}\ov{H}^i\right|=\left|\fint (H^i-\ov{H}^i)|A^i|^2-H^i(H^i-\ov{H}^i)^2d\mu\right|\leq \frac{4}{|\tau|^3}\label{eq:ddtbarHiest}
\end{equation}

Given any $a<b<0$, if $\lambda_i$ is so large that $[a,b]\subset [-\max\{\lambda_i^2(T-T_j), \lambda_i^2T\}, 0)$ then we have that there is a constant $C_1=C_1(a,b)$ such that
\[\ov{H}^i<C_1(a,b)j^{-1}\]
where we used the estimate \eqref{eq:ddtbarHiest}. Therefore, for any smoothly converging subsequence, $\ov{H}^i\ra 0$ uniformly on $[a,b]$ and so the flow converges smoothly on $[a,b]$ to a solution to MCF,
\[\ddt{X^\infty} = -H^\infty \nu^\infty\ .\]
Note that the above argument implies that for \emph{any} subsequence such that the blow up sequence converges in $C^2$, we still get a mean curvature flow.

The argument for the type II blow up is similar: Suppose that we have a converging subsequence of a Type II blow up. Then, defining $T_j$ as in \eqref{eq:Hjest} then substituting $t=t_i+\lambda_i^{-2}\tau$
\[j^{-1}=\int^T_{T_j} \ov{H}^2(t) dt \geq \int^{\l_i^2(T-t_i-i^{-1})}_{\lambda^2_i(T_j-t_i)}\l^{-2}_i\ov{H}^2(\tau)d\tau=\int^{\l_i^2(T-t_i-i^{-1})}_{\lambda^2_i(T_j-t_i)}\ov{H}^2_i(\tau)d\tau\ .\]
By the type II hypothesis we know that $\l_i^2(T-t_i-i^{-1})\ra \infty$, see for example \cite[p.~89]{Mantegazza}, while we may assume that for $i$ large enough $T_j<t_i$ and so $\l_i^2(T_j-t_i)\ra -\infty$. This time, by our choice of blow up sequence, $|A^i(x,t)|\leq 1$ on the time interval $[-\lambda_i^2 t_i,  \lambda_i^2(T-i^{-1}-t_i))$ and so, estimating similarly to \eqref{eq:ddtbarHiest}, on this time interval $|\ddt{}\ov{H}|\leq 4$. As a result, given any finite time interval $[a,b]\subset \bb{R}$, if $i$ is large enough so that $[a,b]\subset[\lambda_i^2(T_j-t_i), \lambda_i^2(T-t_i-i^{-1})]$ then there exists a $C_2=C_2(b-a)$ such that $\ov{H}^i<C_2 j^{-1}$. As we know that the Type II sequence converges locally smoothly, we see that on $[a,b]$ we $\ov{H}^i\ra 0$ uniformly and again the limit is a solution to MCF.
\end{proof}

\begin{lemma} \label{lem:homothetic}
  Suppose that a singularity of VPMCF occurs at a finite time $T$ which is of type~I. 
 Then the asymptotic limit flow from the above theorem is a (possibly trivial) homothetically shrinking solution of the Mean Curvature Flow.
\end{lemma}

\begin{proof}
The proof goes as in the case of type I singularities of the Mean Curvature Flow. We sketch the proof for the convenience of the reader. Note that the Monotonicity formula implies 
\begin{align*}
 \frac{d}{dt}\left(\psi(t)\int_{M_t}\Phi_{(x_0,T)}(x,t)\,d\mu^n_t\right) \leq 0,
\end{align*}
thus the limit $\lim_{t\to T}\Lambda_{(x_0,T)}(t)$ exists, where
$$\Lambda_{(x_0,T)}(t):=\psi(t)\int_{M_t}\Phi_{(x_0,T)}(x,t)\,d\mu^n_t\,.$$
By parabolic rescaling we have that
\begin{align*}
\Lambda_{(x_0,T)}(t)
&=\psi(t)\int_{M_{t}}\Phi_{(x_0,T)}(x,t)\,d\mu_{t} \notag=\psi_i(\tau)\int_{M^i_\tau}\Phi_{(0,0)}(y,\tau)\,d\mu^i_{\tau}
=:\Lambda_{(0,0)}^i(\tau)\,,
\end{align*}
where $t = T + \frac{\tau}{\lambda_i^2}$.
We apply the rescaled monotonicity formula and estimate for $\tau_1<\tau_2$
\begin{align*}
0&\leq\frac{1}{2}\int_{\tau_1}^{\tau_2}\psi_i(\tau) \int_{M^i_\tau}\left|H^i+\frac{\langle y,\nu\rangle}{2\tau}\right|^2\Phi_{0,0}\,d\mu^n_\tau d\tau\\
&\leq\frac{1}{2}\int_{\tau_1}^{\tau_2} \psi_i(\tau)\int_{M^i_\tau} \left\{ \left|H^i - \ov{H}^i+\frac{\langle y,\nu\rangle}{2\tau}\right|^2
		+\left|H^i+\frac{\langle y,\nu\rangle}{2\tau}\right|^2\right\}
				\Phi_{0,0}\,d\mu^n_\tau d\tau
\leq\Lambda_{(0,0)}^i(\tau_1)-\Lambda_{(0,0)}^i(\tau_2)&\notag \\
&=\Lambda_{(x_0,T)}\!\left(T+\frac{\tau_1}{\lambda_i^2}\right)
			-\Lambda_{(x_0,T)}\!\left(T+\frac{\tau_2}{\lambda_i^2}\right)&
\end{align*}
for all $i$.
Since 
$$T+\frac{\tau_l}{\lambda_i^2}\to T$$ 
for $i\to\infty$ and $l=1,2$, the right-hand side of the above converges to 0 for $i\to\infty$. 
The flows $\big(\big(M^i_\tau\big)_{\tau\in[\tau_1,\tau_2]}\big)_{i\in\mathbb{N}}$
converge smoothly along a subsequence and on compact subsets of $\R^{n+1}$ to a smooth flow $\left(M^\infty_\tau\right)_{\tau\in[\tau_1,\tau_2]}$.
Let $R>0$.
There exists a $i_0\in\mathbb{N}$ so that for all $i\geq i_0$, $M_\tau^i\cap B_R(0)$ can be parametrized over $M_\tau^\infty\cap B_R(0)$. 
That is, there exist immersions $Y_i:M_\tau^\infty\cap B_R(0)\to\R^{n+1}$ with
$$M_\tau^i\cap B_R(0)=Y_i(M_\tau^\infty\cap B_R(0))$$
and $Y_i\to\operatorname{id}$ for $i\to\infty$. From the above we have that
\begin{align*}
\Lambda_{(x_0,T)}\!\left(T+\frac{\tau_1}{\lambda_i^2}\right)
			-\Lambda_{(x_0,T)}\!\left(T+\frac{\tau_2}{\lambda_i^2}\right)&\geq\frac{1}{2}\int_{\tau_1}^{\tau_2}\psi_i(\tau) \int_{M^i_\tau\cap B_R(0)}\left|H^i+\frac{\langle y,\nu\rangle}{2\tau}\right|^2\Phi_{0,0}\,d\mu^n_\tau d\tau\ .
\end{align*}
Taking a $\liminf$ and applying Fatou's lemma, 
\begin{align*}
0&=\liminf_{i\to\infty} \int_{\tau_1}^{\tau_2} \int_{M^i_\tau\cap B_R(0)}
		\frac{\psi_i(\tau)}{2}\left|H^i+\frac{\langle y,\nu^i\rangle}{2\tau}\right|^2
				\Phi_{0,0}\,d\mu^i_{\tau}d\tau \notag\\
				&=\liminf_{i\to\infty} \int_{\tau_1}^{\tau_2} \int_{M^i_\tau\cap B_R(0)}
		\frac{\psi_i(\tau)}{2}\left|H^i+\frac{\langle y,\nu^i\rangle}{2\tau}\right|^2
				\Phi_{0,0}\,\sqrt{\det(DY_i)}\,dxd\tau \notag\\
&\geq \int_{\tau_1}^{\tau_2}\int_{M^\infty_\tau\cap B_R(0)}
			\liminf_{i\to\infty}\left( \frac{\psi_i(\tau)}{2}
			\left|H^i+\frac{\langle Y_i,\nu^i\rangle}{2\tau}\right|^2
				\Phi_{0,0}\sqrt{\det(DY_i)}\right)\,dx d\tau \notag\\
&= \frac{1}{2} \int_{\tau_1}^{\tau_2}\int_{M^\infty_\tau\cap B_R(0)}
			\left|H^\infty+\frac{\langle y,\nu^\infty\rangle}{2\tau}\right|^2
				\Phi_{0,0}\,d\mu^\infty_{\tau}d\tau \,,
\end{align*}
where we have used that $\ov{H}^i \to 0$ uniformly as shown in the proof of Theorem~\ref{thm:asymtflow}, and thus $\psi_i(\tau) \to 1$.
Since $R>0$ was arbitrary, we deduce
$$\int_{\tau_1}^{\tau_2}\int_{M^\infty_\tau}\left|H^\infty+\frac{\langle y,\nu^\infty\rangle}{2\tau}\right|^2\Phi_{0,0}\,d\mu_\tau d\tau=0\,.$$
Since the convergence is smooth, and sending $\tau_1\to-\infty$ and $\tau_2\to0$ yields
$$\left|H^\infty+\frac{\langle y,\nu^\infty\rangle}{2\tau}\right|^2=0$$ 
for every $\tau\in(-\infty,0)$ and every $y\in M^\infty_\tau$, so $M^\infty_\tau$ is a homothetically shrinking (possibly trivial) solution of the Mean Curvature Flow.
 
\end{proof}

\subsection{Density estimates and the clearing out lemma} \label{subsec:clearingout}

In this section we demonstrate that the monotonicity arguments in \cite[Chapter 4]{Eckerbook} may be modified to the case of VPMCF with essentially minor additional assumptions depending only on the $L^2$ norm of $\ov{H}$. Throughout this section we will assume that our flow is smooth and properly immersed up to the final, possibly singular, time.

From the Monotonicity formula we have that
\begin{flalign*}
\ddt{} \left( \psi\int_{M_t} \phi\Phi_{(x_0,t_0)} d \mu \right)&\leq  \psi\int_{M_t} \Phi_{(x_0,t_0)}\ho \phi d \mu\ .
\end{flalign*}
and from Corollary \ref{cor:L2Hbound} we have that $1=\psi(0)\geq \psi(t)\geq \e(t)>0$ for some strictly positive function $\e(t)$, and by the dominated convergence theorem, $\psi(t)$ is continuous.

\begin{defses}
The space-time track of a flow is defined by\[\mathcal{M}:=\cup_{t\in[0,T)} M_t\times\{t\} \subset \bb{R}^{n+1}\times [0,T)\ .\] 
The Gaussian density on $\mathcal{M}$ is determined by
\[\T(\mathcal{M}, x_0, t_0):=\lim_{t \nearrow t_0}\int_{M_t} \Phi_{(x_0,t_0) }d\mu_t\]
where, as usual, integration also counts multiplicities in $\mathcal{M}$. 
\end{defses}

Indeed, by the monotonicity formula (with $\phi=1$) we see that
\[\int_{M_t} \Phi_{(x_0,t_0) }d\mu_t\leq \frac{1}{\psi(t)}\int_{M_0} \Phi_{(x_0,t_0) }d\mu_0=: \frac{C_0(M_0)}{\psi(t)}\]
and so by dominated convergence theorem, the density converges everywhere. It is also useful to define a localised density, and so we consider the localisation function
\[\phi_{(x_0,t_0),\rho}(x,t):=\left(1-\frac{|x-x_0|^2+2n(t-t_0)}{\rho^{2}}\right)^3_+\] 
then we have that on the support of $\phi_{(x_0,t_0),\rho}(x,t)$,
\[\ho \phi_{(x_0,t_0),\rho}(x,t)=-\frac{6}{\rho^2}\ov{H}\ip{x-x_0}{\nu}\phi_{(x_0,t_0),\rho}^\frac 2 3\leq \frac{6}{\rho}|\ov{H}|\]
and so by the monotonicity formula 
\[\ddt{} \left( \psi(t)\int_{M_t} \phi_{(x_0,t_0),\rho}\Phi_{(x_0,t_0)} d \mu \right)\leq \frac{6}{\rho}|\ov{H}|\int_{M_t}\psi(t)\Phi_{(x_0,t_0)}d\mu \leq  \frac{6}{\rho}|\ov{H}|C_0(M_0)
\]
where we apply the monotonicity formula to estimate the integral in terms of $M_0$.
In particular, we may estimate that for any $t_2<t_3$,
\begin{flalign}
\psi(t_3)&\int_{M_{t_3}} \phi_{(x_0,t_0),\rho}\Phi_{(x_0,t_0)} d \mu\nonumber\\
&\leq \psi(t_2)\int_{M_{t_2}} \phi_{(x_0,t_0),\rho}\Phi_{(x_0,t_0)} d \mu+6C_0(M_0)\sqrt{\frac{t_3-t_2}{\rho^2}}\sqrt{\int_{t_2}^{t_3} \ov{H}^2dt}\ .\label{eq:AlmostMonotonicity}
\end{flalign}
With some abuse of notation, we define localised monotonicity to be
\[\T(\mathcal{M}, x_0, t_0):=\lim_{t \nearrow t_0}\int_{M_t} \phi_{(x_0,t_0),\rho}\Phi_{(x_0,t_0) }d\mu_t\ ,\]
and note that, by \eqref{eq:AlmostMonotonicity}, this limit always exists and we may estimate that for $t\in(t_0-\rho^2, t_0)$
\begin{equation}\T(\mathcal{M}, x_0, t_0)\leq \frac{\psi(t)}{\psi(t_0)}\int_{M_{t}} \phi_{(x_0,t_0),\rho}\Phi_{(x_0,t_0)} d \mu+6\frac{C_0(M_0)}{\psi(t_0)}\sqrt{\frac{t_0-t}{\rho^2}}\sqrt{\int_{t}^{t_0} \ov{H}^2dt}.\label{eq:LocalDenistyest}
\end{equation}
\begin{lemma}[Upper semicontinuity of $\Theta$]
$\T(\mathcal{M}, x_0, t_0)$ is upper semi continuous in time and space. Explicitly, if $x_j \ra x_0$, $t_j \nearrow t_0$ then 
\[\limsup_{j\ra \infty} \T(\mathcal{M}, x_j, t_j)\leq \T(\mathcal{M}, x_0, t_0)\ .\]
\end{lemma}
\begin{proof}
We have that for any $t\in (t_0-\rho^2,t_0)$, if $j$ is large enough then
\[\T(\mathcal{M}, x_j, t_j)\leq \frac{\psi(t)}{\psi(t_j)}\int_{M_{t}} \phi_{(x_j,t_j),\rho}\Phi_{(x_j,t_j)} d \mu+6\frac{C_0(M_0)}{\psi(t_j)}\sqrt{\frac{t_j-t}{\rho^2}}\sqrt{\int_{t}^{t_j} \ov{H}^2dt}\ .
\]
so, taking a limsup,
\[\limsup_{j\ra\infty}\T(\mathcal{M}, x_j, t_j)\leq \frac{\psi(t)}{\psi(t_0)}\int_{M_{t}} \phi_{(x_0,t_0),\rho}\Phi_{(x_0,t_0)} d \mu++6\frac{C_0(M_0)}{\psi(t_0)}\sqrt{\frac{t_0-t}{\rho^2}}\sqrt{\int_{t}^{t_0} \ov{H}^2dt}
\]
and limiting $t\ra t_0$ yields the required inequality (as $\psi$ is continuous).
\end{proof}
\begin{cor}\label{cor:Densityabove1}
For any point $(x_0,t_0)$ reached by a smooth properly immersed flow,
\[\T(\mathcal{M},x_0,t_0)\geq 1\]
\end{cor}
\begin{proof}
On any smooth manifold with $x_j\in M_{t_j}$, $\T(\mathcal{M},x_j,t_j)=1$, so the statement follows from upper semicontinuity.
\end{proof}

\begin{proposition}[The clearing out lemma]\label{prop:clearingout}
Suppose that $M_t$ is smooth and properly immersed for $t\in [t_0-\rho_0^2,t_0)$. We suppose that for positive constants $C_0$ and $L$
\[\int_{t_0-\rho_0^2}^{t_0} \ov{H}^2 dt\leq L \qquad \psi(t_0-\rho_0^2)\int_{M_{t_0-\rho^2_0}} \Phi_{(x_0,t_0)} d\mu\leq C_0\ .\]
Then there exists a constant $\beta_0=\beta_0(n,C_0,L)$ such that for any $\beta\in(0,\beta_0)$ there exists a constant $\theta(\beta, n)$ such that for any $\rho \in (0,\rho_0)$
\[\frac{\mathcal{H}^n(M_{t_0-\beta\rho^2}\cap B_{\rho})}{\rho^n}\geq \theta e^{-^\frac{L}{2}}\ .\]
\end{proposition}
\begin{proof}
 By Corollary \ref{cor:Densityabove1} and \eqref{eq:LocalDenistyest} we have that for $t\in(t_0-\sigma^2,t_0)$ where $\sigma\in(0,\rho_0)$,
\[1\leq \frac{\psi(t)}{\psi(t_0)}\int_{M_{t}} \phi_{(x_0,t_0),\sigma}\Phi_{(x_0,t_0)} d \mu+6\frac{C_0(M_0)}{\psi(t_0)}\sqrt{\frac{t_0-t}{\sigma^2}}\sqrt{\int_{t}^{t_0} \ov{H}^2dt}\ .\]
Rewriting, and setting $t_0-t=\alpha \sigma^2$ for some $\alpha\in(0,1]$,
\begin{flalign*}1&-6\frac{C_0(M_0)}{\psi(t_0)}\sqrt{\frac{t_0-t}{\sigma^2}}\sqrt{\int_{t}^{t_0} \ov{H}^2dt}\\
&\leq \frac{\psi(t)}{\psi(t_0)}\int_{M_{t}} \left(1-\frac{|x-x_0|^2+2n(t-t_0)}{\sigma^{2}}\right)^3_+\frac{1}{(4\pi(t_0-t))^\frac n2}e^{-\frac{|x-x_0|^2}{4(t_0-t)}} d \mu\\
&\leq \frac{\psi(t)}{\psi(t_0)}\frac{(1+2n\alpha)^3}{(4\pi \alpha \sigma^2)^\frac n 2}\mathcal{H}^n(M_{t_0-\alpha \sigma^2}\cap B_{\sqrt{1+2n\a} \sigma}(x_0))\ .
\end{flalign*}
We wish to change constants from $(\alpha, \sigma)$ to $(\beta, \rho)$ by setting $\rho= \sqrt{1+2n \alpha} \sigma$ and $\beta \rho^2 = \alpha \sigma^2$. For this to be possible, we need $\beta<\frac 1 {2n}$ at which point we have that 
\begin{align*}
\beta=\frac{\a}{1+2n\a}&&\a=\frac{\beta}{1-2n\beta}&&\sigma=\sqrt{1-2n\beta}\rho
\end{align*}
 Therefore,
\begin{flalign*}
\frac{(4\pi \beta)^\frac n 2}{(1+2n\alpha)^3}\left(1-6\sqrt{\alpha}C_0\sqrt{L}e^\frac{L}{2}\right)&\leq e^\frac{L}{2}\frac{\mathcal{H}^n(M_{t_0-\beta \rho^2}\cap B_{\rho})}{\rho^n}\ .
\end{flalign*}
To get a positive density we want $\a= \min\{\frac{e^{-L}}{144C_0^2 L},1\}$ which corresponds to \[\beta\leq \min\left\{\frac{1}{144C_0^2Le^{L}+2n}, \frac{1}{1+2n}\right\}=:\beta_0\ .\] Finally we obtain,
\begin{flalign*}
\t(n,\beta)e^{-^\frac{L}{2}}:=\frac 1 2(4\pi \beta)^\frac n 2(1-2n\beta)^3e^{-^\frac{L}{2}}&\leq \frac{\mathcal{H}^n(M_{t_0-\beta \rho^2}\cap B_{\rho}(x_0))}{\rho^n}\ .
\end{flalign*}
\end{proof}

\subsection{A condition for a uniform diameter bound}

In this section we apply the clearing out lemma to obtain conditions under which we may ensure a uniform curvature bound.

\begin{proposition}\label{prop:uniformdiambound}
Suppose that there exists constants $h, C>0$ such that for all $t\in[0, T-h)$
\begin{equation}\int_t^{t+h}\ov{H}^2 dt <C\ . \label{eq:conditionfordiambound}\end{equation}
Then there exists a constant $R=R(C, n, M_0)$ such that for all $t\in [0,T)$
\[\diam(M_t)<R\ .\] 
\end{proposition}
\begin{proof}
We apply the clearing out lemma (Proposition \ref{prop:clearingout}) with $\rho_0=\sqrt{h}$ on any time interval of the form $[t,t+h)$. On such an interval by assumption we have the $L^2$ bound on $\ov{H}$, and the second assumption is automatically fulfilled for some $C_0=C_0(M_0)$ by the monotonicity formula. Therefore we have, for $\beta=\frac{\beta_0(n, C, M_0)}{2}$ there is a $\tilde{\theta}=\theta(\beta, n)e^{-\frac C 2}$ such that $\mathcal{H}^n(M_{t_0-\beta\rho^2}\cap B_{\rho})\geq \tilde{\theta}\rho^n$. 

 Suppose that $t>h\beta$. We have that $|M_t|\leq |M_0|$ and there exists $\frac{\rho_0}4\leq \rho<\rho_0$ such that $\diam(M_t)=2N\rho$ for $N\in \mathbb{N}$. Then there must be at least $N$ disjoint ambient balls of radius $\rho$ whose centers lie on $M_t$. We use the area of $M_t$ in these balls as lower bounds. We see that (as $t>\beta \rho^2$) \[|M_0|>|M_{t-\beta\rho^2}|>N\tilde \theta \rho^n=\tfrac{1}{2}\tilde\theta \rho^{n-1}\diam (M_t) .\]
As $\rho\geq \frac {\rho_0} 4$ and $\tilde\theta$ is fixed, we see that $\diam (M_t)\leq 2 \tilde\theta^{-1}\rho_0^{1-n}4^{n-1}|M_0|$.\\
If $t<h \beta$, we may instead apply Proposition \ref{prop:dtdiam} to see that  $\diam (M_t)\leq C(M_0, h, \beta)$.

\end{proof}
\begin{remark}
The condition \eqref{eq:conditionfordiambound} in the above Proposition may also be replaced by \\
$\int_t^{t+h}\int_{M_t} H^2d\mu dt <C$ or equivalently $|\ddt{}|M_t||<C$. Clearly a uniform bound on $\ov{H}$ also implies this condition. 
\end{remark}

\section{Long time behaviour of VPMCF}\label{sec:longtime}

In this section we consider the flow under the assumptions that 
\begin{enumerate}[label={\alph*)}]
\item the flow exists for all time, and
\item the diameter is uniformly bounded, that is, there exists an $R_0>0$ such that
\begin{equation}
\diam (M_t)\leq R_0 \label{eq:diambound}
\end{equation}
for all $t\in[0,\infty)$.
\end{enumerate}
By Proposition \ref{prop:uniformdiambound} we can replace \eqref{eq:diambound} with \eqref{eq:conditionfordiambound}. Both of the above assumptions are therefore implied by a curvature bound on the flow which is uniform in time.

\begin{defses}
We define a ``Newton inequality-like'' extended isoperimetric ratio by the scale invariant quantity
\[\mathcal{I}(M_t) = \frac{n+1}{n}\ov{H}\frac{\operatorname{Vol}^{n+1}(M_t)}{|M_t|}\ .\]
\end{defses}
\begin{remark}
This isoperimetric ratio appears to be a natural one to consider: If we think of the Alexandrov--Fenchel inequalities as an ``averaged'' version of the Mclaurin inequalities for symmetric polynomials, then the above would be a ratio similar to an ``averaged'' Newton inequalities. Importantly, the Newton inequalities hold without any assumptions on the curvature cone.
\end{remark}
\begin{remark} \label{rem:AlexFenchel}
Denoting $\omega_{n}$ for the area of the $n$-sphere, if we write $v_{n+1}=\frac{n+1}{\omega_{n}}\on{Vol}^{n+1}(M_t)$, $v_n=|M_t|$ and $v_{n-1}=\frac{1}{n}\int_{M_t} H d\mu$ then, while $M_t$ is mean convex and starshaped, then it bounds some domain $\Omega_t$ and we know that, by the Alexandrov--Fenchel inequalities, $v_{n+1}^\frac{1}{n+1}\leq v_{n}^\frac{1}{n}\leq v_{n-1}^\frac{1}{n-1}$ with equalities if and only if $M_t$ is a sphere (see \cite{GuanLiQuermass} -- here we are writing, $v_{n+1-k}=\frac{V_{n+1-k}(\Omega_t)}{V_{n+1-k}(B)}$ in that paper). Applying these we have that
\[1\leq \frac{v_{n+1}}{v_n^\frac{n+1}{n}}\leq \mathcal{I}(M_t)=\frac{v_{n-1}v_{n+1}}{v_n^2}\leq \frac{v_{n-1}}{v_n^\frac{n-1}{n}}=\frac{1}{\mathcal{I}_1(\Omega_t)^{n-1}}\ ,\] 
where $\mathcal{I}_1(\Omega_t)$ is the first Quermass integral 
\cite[equation (6)]{GuanLiQuermass}. Equalities hold in the abpve if and only if $M_t$ is a sphere. Therefore if $M_t$ is starshaped and mean convex with $\mathcal{I}(M_t)=1$, $M_t$ is a round sphere.
\end{remark}
Under the above assumptions we may improve the assumptions of section \ref{sec:diamL2} to get the following:
\begin{theorem}\label{thm:asymptotics}
Suppose that $M_t$ satisfies VPMCF for all $t>0$ with $V_0\neq 0$ and has uniformly bounded diameter as in \eqref{eq:diambound}. Then there exists a subsequence $t_i \ra \infty$ such that either
\[\mathcal{I}(M_{t_i})\ra 1 \qquad \text{ or }\qquad \mathcal{I}(M_{t_i})\ra 0\]as $i\ra\infty$. Equivalently, there exists a subsequence such that is the extended isoperimetric ratio goes to that of a sphere or of a minimal surface.
\end{theorem}
\begin{cor}\label{cor:meanconvexstartshaped}
Any mean convex starshaped solution of VPMCF with uniformly bounded curvature exists for all times and converges to a sphere.
\end{cor}

\begin{proof}
Of Corollary~\ref{cor:meanconvexstartshaped}.
First we remark that the uniform bound on $|A|$ implies uniform bounds on all derivatives of $A$ and thus the flow subsequentially converges smoothly to a limiting object. Using \cite{GuanLiQuermass} as in Remark~\ref{rem:AlexFenchel} we know that $\mathcal{I}(M_t) \geq 1$ along the flow. So the above theorem implies that the limiting object has extended isoperimetric ratio equal to one and must be the round sphere. By \cite{EscherSimonett} we improve subconvergence to convergence.
\end{proof}

\begin{remark}
\begin{itemize}
 \item  If we have a VPMCF that stays Alexandrov immersed for all times and $|A|\leq C$ uniformly, then for each $t_j\to\infty$, the flow subconverges smoothly to a closed limiting surface with constant mean curvature ($H(\cdot,\tilde t_j)$ must converge to $\ov{H}_\infty$ see e.g.\ \cite{HuiskenConvexVPMCF}). In this Alexandrov immersed case the limiting object must be a sphere -- and thus $\mathcal{I}(M_\infty)=1$. This follows from Alexandrov's theorem: A closed Alexandrov immersion having constant mean curvature must be a sphere \cite{Alexandrov}. We can also improve subconvergence to convergence by \cite{EscherSimonett} again. 
 \item In general, if no singularity appears, i.e.\ $T=\infty$ and $|A|\leq C$ (uniformly in $t$), then $\mathcal{I}(M_{t_j})\to 0$ cannot happen because any limiting surface satisfies $|M_\infty|>0$, so  $\mathcal{I}(M_{\tilde t_j})\to 0$ would imply we found a closed minimal surface as limiting object which is clearly impossible. At this point we do not know whether $  \mathcal{I}(M_{\infty})=1$ implies that $   M_{\infty}$ is a sphere. We only know this for the starshaped case and for Alexandrov immersions.   

\end{itemize}

\end{remark}

In order to show Theorem~\ref{thm:asymptotics} we first improve our estimates on $\ov{H}^2$ in the two cases of $\mathcal{I}(M_t)\leq 1$ (where $\mathcal{I}(\bb{S}^n)=1$) and $\mathcal{I}(M_t)\geq 1$.
\begin{proposition}\label{prop:lessthaniso}
Suppose that $M_t$  as in \eqref{eq:diambound} and $V_0\neq 0$. Let $\a\in (0,1)$, and suppose that on the time interval $[t_1, t_2]$ we may estimate
\[\ov{H}(n+1)V_0\leq \alpha n |M_t|\ .\]
Then we have that
\begin{align*}
 \int_{t_1}^{t_2}|\ov{H}|^2dt
 &\leq \max\left(\frac{\alpha^2}{(1-\alpha)^2}, 1\right)\frac{R^2_0
 }{2(n+1)^2V_0^2}|M_{t_1}|^2\ .
\end{align*}
Note that $\frac{\alpha^2}{(1-\alpha)^2 }\geq 1$ if and only if $\alpha\geq\frac{1}{2}$. 
\end{proposition}
\begin{proof}
In Proposition \ref{prop:ElenasArgument2}, we estimated $\ov{H}$ using Young's inequality. Here we estimate more carefully using the above assumption.

Using \eqref{eq:HsuppisArea}, we note that,
\begin{flalign}
\begin{split}
\int_{M_t} \ip{X_t}{X-x}d \mu&=\ov{H}\int_M\ip{\nu}{X-x}d \mu-\int_M H\ip{\nu}{X-x}d \mu \\
&=\ov{H}(n+1)V_0-n|M_t| \\
&\leq - n(1-\a) |M_t|
\end{split}\label{eq1}
\end{flalign}
We set $b=n|M_t|$ and $a=-\int_{M_t} \ip{X_t}{X-x}d \mu\geq 0$. Then we consider two cases. First, we assume that $a<b$. We use $a= \alpha a + (1-\alpha) a \geq (1-\alpha)b$ which implies $|b-a|=b-a \leq \frac{\alpha}{1-\alpha} a = \frac{\alpha}{1-\alpha} |a|$. On the other hand, if $a>b$, then we know that $|b-a| = a-b < a=|a|$ because $b>0$. Putting this together, we get that $|b-a|\leq \max(1,\frac{\alpha}{1-\alpha})|a|$ and thus
\[\left|\int_{M_t} \ip{X_t}{X-x}d \mu+n|M_t|\right|\leq \max\left(\frac{\alpha}{1-\alpha},1\right) \left|\int_M \ip{X_t}{X-x}d \mu \right| \]
and 
\[|\ov{H}|\leq  \max\left(\frac{\alpha}{1-\alpha},1\right)\frac{\left|\int_M \ip{X_t}{X-x}d \mu \right|}{(n+1) |V_0|}\ .\]

Estimating as before we have
\begin{align*}
 |\ov{H}|^2
 &\leq \max\left(\frac{\alpha^2}{(1-\alpha)^2},1\right)\frac{\int_M |H-\ov{H}|^2d \mu\int_M |X|^2d \mu}{(n+1)^2V_0^2}\\
 &\leq -\max\left(\frac{\alpha^2}{(1-\alpha)^2},1\right)\frac{R_0^2}{2(n+1)^2V_0^2}\ddt{}|M_t|^2\ .
\end{align*}
And so
\[\int_{t_1}^{t_2}|\ov{H}|^2dt \leq  \max\left(\frac{\alpha^2}{(1-\alpha)^2},1\right)\frac{R^2_0}{2(n+1)^2V_0^2}[|M_{t_1}|^2-|M_{t_2}|^2]\ .\]

\end{proof}

We may now repeat this but for the opposite inequality of the (extended) isoperimetric(/Alexandrov-Fenchel) inequality.
\begin{proposition}\label{prop:morethaniso}
Suppose that $M_t$ has $\operatorname{diam}(M_t)\leq R$ and $V_0\neq 0$. If on the time interval $[t_1, t_2]$ we may estimate
\[\ov{H}(n+1)V_0\geq (1+\beta) n |M_t|\ ,\]
for some  $\beta>0$, then we have that
\[\int_{t_1}^{t_2}|\ov{H}|^2 dt 
\leq \frac{R^2(1+\beta^{-1})^2}{2(n+1)^2V_0^2}|M_{t_1}|^2 \ .\]
\end{proposition}

\begin{proof}
We note that from the assumption,
\begin{flalign*}\int_{M_t} \ip{X_t}{X-x}d \mu&=\ov{H}(n+1)V_0-n|M_t| \\
&\geq  \beta n |M_t|\ .
\end{flalign*}
This time, may directly estimate using the triangle inequality that
\[|\ov{H}|\leq \frac{(1+\beta^{-1})\left|\int_M \ip{X_t}{X-x}d \mu \right|}{(n+1)|V_0|}\ .\]
so, as above
\[|\ov{H}|^2\leq -\frac{R^2_0(1+\beta^{-1})^2}{2(n+1)^2V_0^2}\ddt{}|M_t|^2\ ,\]
and so 
\[\int_{t_1}^{t_2}|\ov{H}|^2 dt \leq \frac{R^2_0(1+\beta^{-1})^2}{2(n+1)^2V_0^2}[|M_{t_1}|^2-|M_{t_2}|^2]\ .\]
\end{proof}

\begin{proof}[Proof of theorem \ref{thm:asymptotics}]
Suppose not. Then there exists a $T, \e>0$ such that for all $t>T$, $\mathcal{I}(M_t)\in(-\infty, -\e)\cup(\e,1-\e)\cup(1+\e,\infty)$. By continuity we have that $\mathcal{I}(M_t)$ is in one of the intervals $(-\infty, -\e)$, $(\e,1-\e)$ or $(1+\e,\infty)$ for all $t>T$. 
We particularly have that $\ov{H}(n+1) V_0 = n|M_t| \mathcal{I}(M_t)$ is either not larger than $(1-\epsilon)n|M_t|$ or not less than $(1+\epsilon)n|M_t|$. 
As a result, by applying Propositions \ref{prop:lessthaniso} and \ref{prop:morethaniso} we have one of
\[\int_T^\infty \ov{H}^2dt \leq \max\left(\frac{\epsilon^2}{(1-\epsilon)^2},1\right) \frac{R^2_0}{2(n+1)^2V_0^2}|M_{t_1}|^2\ .\]
or
\[\int_T^\infty \ov{H}^2dt \leq \frac{R^2_0(1+\e^{-1})^2}{2(n+1)^2V_0^2}|M_{t_1}|^2 \ .\]
Using the isoperimetric inequality, we have that $|\ov{H}|= \frac{n|M_t|}{(n+1)|V_0|}|\mathcal I(M_t)|\geq c>0$ in both cases. This is clearly impossible.

\end{proof}

\begin{cor}
If a solution exists for all $t\geq 0$ and for all this time $\ov{H}(n+1)V_0\leq  \alpha n|M_t|$ for some $\alpha\in(0,1)$ then there exist $t_i$ such that $\ov{H}(t_i)\ra 0$ as $i\ra \infty$.
\end{cor}
\section{Barriers and preservation of Alexandrov immersions} \label{sec:Alex}

In Corollary~\ref{cor:Alexhleq0} we already saw that the property of being Alexandrov immersed \underline{and} having $\ov{H}\leq 0$ is not preserved for long under the VPMCF. Now we will study under which conditions the property of being Alexandrov immersed is preserved for all times. We will also describe how barriers for the VPMCF can be constructed. \\

We start by introducing some notation. We will use descriptive superscripts to distinguish between manifold, e.g.\ $M^\text{outer}_t$, $M^\text{Inner}$ and so on. For example, we let $f^\ast:[0,T) \ra \bb{R}$ be some smooth function. 
 
\begin{defses}
We will say a flow $M_t^\ast$ satisfies \emph{Forced Mean Curvature Flow} with forcing term $f^\ast$ (and write this as FMCF($f^\ast$)) if it has a time dependent parametrisation $X^\ast:N^n \times [0,T)\ra\bb{R}^{n+1}$ (and chosen normal $\nu^\ast$) so that
\[\ip{\ddt{X^\ast}+H^\ast\nu^\ast}{\nu^\ast}= f^\ast\ .\]
\end{defses}
We will use this to produce barriers for the flow.

\begin{defses}
Suppose that $M_t^1$ satisfies FMCF($f^1$) and $M_t^2$ satisfies FMCF($f^2$). Suppose that there is an open ball $B_{r}(x)$ so that $M_t^1$ has non-trivial intersection with $B_{r}(x)$ so that $B_{r}(x)\setminus M_t^1$ is made up of two connected components. 

We define the component which the normal points away from to be the \emph{local interior} of $M_t^1$, and the component the normal points into to be the \emph{local exterior} of $M_t^1$. 

We will say that $M^2_t$ is \emph{locally inside} $M_t^1$ if it is contained in the local interior of $M_t^1$. We will say that $M^2_t$ is \emph{locally outside} $M_t^1$ if it is contained in the local exterior of $M_t^1$
\end{defses}

\begin{lemma}\label{lem:InnerOuterLocal}
Suppose that $M_t^\text{out}$ satisfies FMCF($f^\text{out}$) and $M_t^\text{in}$ satisfies FMCF($f^\text{in}$). Suppose that in open some ball $B_{2r}(x)$, $M^\text{in}_0$ is locally inside $M^\text{out}_0$. Given $\sigma \in\{1,-1\}$, suppose that
\[f^\text{out}-\sigma f^\text{in}\geq 0\ .\]
Then $M^\text{in}_t\cap B_{r}(x)$ cannot intersect $M^\text{out}_t\cap B_{r}(x)$ for the first time with $\ip{\nu^\text{in}}{\nu^\text{out}}=\sigma$ at an interior point.
\end{lemma}

\begin{proof}
\underline{Step 1}: The graph equation. We describe the evolution of a graphical piece of a surface moving by FMCF($f^\text{gr}$):\\
 We consider a solution $M^\text{gr}_t$ to FMCF($f^\text{gr}$) locally as a graph of the function graph function $u^\text{gr}$ in direction $e_{n+1}$ (possibly after rotation). Then we calculate with respect to unit normal $\mu^\text{gr}$ showing upwards in the graphical setting (the graphical normal $\mu^\text{gr}$ may not be the same normal as the a priori chosen $\nu^\text{gr}$)
 \begin{align*}
 & v^\text{gr}=\sqrt{1+|Du^\text{gr}|^2}, \qquad \mu^\text{gr} = v^{-1}_\text{gr}(-Du^\text{gr} + e_{n+1}),\\
  & g_{ij} = \delta_{ij} + D_i u^\text{gr} D_j u^\text{gr},\qquad g^{ij} = \delta^{ij} -v^{-2}_\text{gr} D^i u^\text{gr} D^j u^\text{gr},
 \end{align*}
and
\[ h_{ij}^\text{gr} =-v^{-1}_\text{gr} D^2_{ij} u^\text{gr}, \qquad H=-v^{-1}_\text{gr}(\delta^{ij} -v^{-2}_\text{gr} D^i u^\text{gr} D^j u^\text{gr})D^2_{ij} u^\text{gr}\]
and so
\[\ip{u^\text{gr}_t e_{n+1} - v^{-1}_\text{gr}(\delta^{ij} -v^{-2} D^i u^\text{gr} D^j u^\text{gr})D^2_{ij} u^\text{gr} \nu^\text{gr}}{\nu^\text{gr}}=f^\text{gr}\ip{\nu^\text{gr}}{\mu^\text{gr}} \]
or
\[u_t^\text{gr} = (\delta^{ij} -v^{-2}_\text{gr} D^i u^\text{gr} D^j u^\text{gr})D^2_{ij} u^\text{gr} + v^\text{gr} f^\text{gr}\ip{\nu^\text{gr}}{\mu^\text{gr}}\ .\]
Here (while hypersurface is graphical), $\ip{\nu^\text{gr}}{\mu^\text{gr}}\in\{1,-1\}$. \\[0.1cm]

\noindent\underline{Step 2}: Proof of the statement.\\
Suppose not. Suppose that the first point of intersection is $(p_0, t_0)$ for some $t_0>0$. Then after rotation we may assume that $\nu^\text{out}(p_0,t_0)=e_{n+1}$ and so $\nu^\text{in}(p_0,t_0)=\pm e_{n+1}$. For some $\e>0$ we can write both $M^\text{out}_t$ and $M^\text{in}_t$ graphically on $B_\e(p_0)\times(t_0-\e,t_0]$ where 
\[u_t^\text{out} = (\delta^{ij} -v^{-2}_\text{out} D^i u^\text{out} D^j u^\text{out})D^2_{ij} u^\text{out} + v^\text{out} f^\text{out}\ ,\]
and
\[u_t^\text{in} = (\delta^{ij} -v^{-2}_\text{in} D^i u^\text{in} D^j u^\text{in})D^2_{ij} u^\text{in} + v^\text{in} f^\text{in}\ip{\nu^\text{in}}{\mu^\text{in}}\ .\]
Due to our definitions we have that $w:=u^\text{out}-u^\text{in}>0$ for $t<t_0$ and 
\begin{flalign*}
w_t&= a^{ij}D^2_{ij} w +\hat{b}^i D_i w +\sqrt{1+|Du^\text{out}|^2}f^\text{out} - \sqrt{1+|Du^\text{in}|^2}\ip{\nu^\text{in}}{\mu^\text{in}}f^\text{in}\\
&= a^{ij}D^2_{ij} w +\hat{b}^i D_i w +\sqrt{1+|Du^\text{out}|^2}(f^\text{out}-\ip{\nu^N}{\mu^N}f) \\
&\qquad+(\sqrt{1+|Du^\text{out}|^2}- \sqrt{1+|Du^\text{in}|^2})\ip{\nu^\text{in}}{\mu^\text{in}}f^\text{in}\\
&= a^{ij}D^2_{ij} w +\hat{b}^i D_i w +\sqrt{1+|Du^\text{out}|^2}(f^\text{out}-\ip{\nu^\text{in}}{\mu^\text{in}}f)\\
&\qquad +\frac{(D^iu^\text{out}+D^iu^\text{in})(D^iu^\text{out}-D^iu^\text{in})}{\sqrt{1+|Du^\text{in}|^2}+ \sqrt{1+|Du^\text{in}|^2}}\ip{\nu^\text{in}}{\mu^\text{in}}f^\text{in}\\
&= a^{ij}D^2_{ij} w +b^i D_i w +\sqrt{1+|Du^\text{out}|^2}(f^\text{out}-\ip{\nu^\text{in}}{\mu^\text{in}}f^\text{in})
\end{flalign*}
where $a^{ij}$ and $b^i$ are the usual terms and we have computed the additional terms to needless accuracy. We know that $\ip{\nu^\text{in}}{\mu^\text{in}}=\ip{\nu^\text{in}(p_0,t_0)}{\nu^\text{out}(p_0,t_0)}=\sigma$. Therefore as $f^\text{out}-\ip{\nu^\text{in}}{\mu^\text{in}}f^\text{in}\geq 0$, we have a contradiction to the strong maximum principle.
\end{proof}

\begin{lemma}\label{lem:OuterLocal}
Suppose that $M_t^1$ satisfies FMCF($f^1$) and $M_t^\text{out}$ satisfies FMCF($f^\text{out}$). Suppose that in a ball $B_r(x)$ we have that $M_0^1$ is locally outside $M_0^\text{in}$. Suppose that for some $\sigma\in\{1,-1\}$
\[f^1-\sigma f^\text{out}\leq 0\]
Then $M^1_t\cap B_{r}(x)$ cannot intersect $M^\text{out}_t\cap B_{r}(x)$ for the first time with $\ip{\nu^1}{\nu^ \text{out}}=\sigma$ at an interior point.
\end{lemma}
\begin{proof}
Suppose not. This time, we may take $e_{n+1} = \nu^1(p_0,t_0)=\sigma\nu^\text{out}(p_0,t_0)$ (as otherwise one would have been locally inside the other). This time $w=u^1-u^\text{out}<0$ and 
\begin{flalign*}
w_t
&= a^{ij}D^2_{ij} w +b^i D_i w +\sqrt{1+|Du^1|^2}(f^1-\sigma f^\text{out})\ .
\end{flalign*}
Again we get a contradiction to the strong maximum principle.
\end{proof}

\begin{proposition}[Outer Barriers]
Suppose that $M^\text{out}_t$ is a solution for FMCF($f^\text{out}$) which is embedded for all time and bounds the compact region $\Lambda(t)$, and $\nu^N$ points out of $\Lambda(t)$. Suppose that $M_t$ satisfies VPMCF and $M_0\subset \overset{\circ}{\Lambda(0)}$ and for all $t\in[0,T)$
\[f^\text{out}\geq |\ov{H}| .\]
Then for all $t\in [0,T)$, $M_t \subset \Lambda(t)$. 
\end{proposition}
\begin{proof}
Suppose not. We have that for $\sigma\in\{1,-1\}$, 
\[f^\text{out} - \sigma \ov{H}\geq f^\text{out}-|\ov{H}|\geq 0\]
and so by Lemma~\ref{lem:InnerOuterLocal}, at a first time of intersection the two normals cannot be multiples of each other, a contradiction.
\end{proof}

\begin{proposition}[Inner Barriers]
Suppose that $N_t$ is a solution for FMCF($f$) which is  embedded for all time and bounds the compact region $\Lambda(t)$, and $\nu^N$ points out of $\Lambda(t)$. Suppose that $M_0\cap \overline{\Lambda(0)}=\emptyset$, and for all $t\in[0,T)$
\[f\leq -|\ov{H}| .\]
Then for all $t\in [0,T)$, $M_t\cap \overline{\Lambda(0)}=\emptyset$. 
\end{proposition}
\begin{proof}
Suppose not. For any $\sigma\in\{1,-1\}$
\[f^1-\sigma \ov{H}\leq f^1+ |\ov{H}|\leq 0\]
and so by Lemma \ref{lem:OuterLocal}, at a first time of intersection the two normals cannot be multiples of each other, a contradiction.
\end{proof}

We need the following corollary from our previous paper:

\begin{cor}[{\cite[Corollary 5]{OurPreviousPaper}}]\label{cor:AlexLoss}
A compact flowing manifold with bounded curvature may only loose the property of being Alexandrov immersed at time $T$ if there exist points $p, q(t)\in M^n$ so that $|X(p,t)-X(q(t),t)|$ goes to zero with $\ip{\nu(p,t)}{\nu(q(t),t)}=-1$ and where $\ip{\nu(p)}{X(q(t),t)-X(p,t)}<0$ and $\ip{\nu(q,t)}{X(p,t)-X(q(t),t)}<0$ for $T-\delta<t<T$.
\end{cor}

\begin{theorem}
 Let $M_0$ be Alexandrov immersed $M_t$ the VPMCF starting from $M_0$. As long as $\ov{H} \geq 0$, the flow stays Alexandrov immersed.
\end{theorem}
\begin{proof}
From Corollary \ref{cor:AlexLoss} we know that at a time where Alexandrov immersion property is lost two disjoint pieces of the flow must intersect, with the normals are in opposite directions. In Lemma \ref{lem:InnerOuterLocal} this corresponds to $\sigma=-1$ while $f^\text{out}=f^{\text{in}}=\ov{H}$, and so we immediately see that this cannot happen.
\end{proof}

A natural question to ask is when $\ov{H}\geq 0$ is preserved if we know that $M_t$ is Alexandrov immersed. This is not always true without further assumptions - see the Example  \ref{ex:trilobite} below.

\appendix
\section{The Trilobite: An example in which $\overline{H}\geq 0$ is lost}\label{appendix1}

A natural hope, given the previous section, is to attempt to show that under Alexandrov immersed VPMCF, $\ov{H}(t)\geq 0$ is preserved. The following example demonstrates that this is not the case without further assumptions on the flow.

The evolution equation for $\ov{H}$ is given by
\[\ddt{\ov{H}} = \fint (H-\ov{H})|A|^2 + (H-\ov{H})(-H^2 +\ov{H}H)d\mu = \fint (H-\ov{H})|A|^2 - (H-\ov{H})^2Hd\mu.\]
So, for surfaces ($n=2$), if $\ov{H}(0)=0$, then 
\begin{align*}
 \ddt{\ov{H}}\bigg|_{t=0} = \fint H (|A|^2 - H^2)d \mu \big|_{t=0} = -2 \fint H K d\mu\big|_{t=0}
\end{align*}
using the Gauss equations. We now describe initial data such that $\int_M H d\mu=0$ and $\int_M KH d\mu >0$, meaning that $\ov{H}\geq 0$ is not preserved.  

Our example is constructed from rotationally symmetric pieces. If $\gamma$ is the profile curve (parametrized by arclength) of a rotationally symmetric surface (rotated about the $y$-axis) then we have that the curvatures are given by
\[\kappa_1 = \kappa, \qquad \kappa_2 =\frac{\ip{\dot{\gamma}}{e_y}}{x}=\frac{\dot{y}}{x}, \qquad H=\frac{\dot{y}}{x}+\kappa, \qquad K=\frac{\dot{y}\kappa}{x}\]
and we may see that $d\mu = 2\pi x ds$.\\
Therefore we have that (for $I=[a,b]$)
\begin{align}\int_M H d\mu  = 2\pi \int_\gamma x\kappa ds +2\pi ( y(b)-y(a)) &&\int_M HK d\mu = 2\pi\int_\gamma \dot{y} \kappa^2 + \kappa \frac{\dot{y}^2}{x} ds \ .\label{eq:HintrsHKintrs}
\end{align}
We calculate the above for various profiles.\\
\noindent\textbf{On a hemisphere $E_\rho$ of radius $\rho$:} We have that $\kappa_1=\kappa_2=\rho^{-1}$ and so in this case $H=2\rho^{-1}$, $K = \rho^{-2}$, 
\[\int_{E_\rho} H = 4\pi \rho, \qquad \int_{E_\rho} HK d\mu = 4\pi \rho^{-1}\ .\]

\noindent\textbf{On a cylinder $C_{\rho, l}$ of radius $\rho$ and length $l$:} We have that $\kappa_1=0$, $\kappa_2=\rho^{-1}$  and so in this case $H=\rho^{-1}$, $K = 0$, 
\[\int_{C_{\rho, l}} H = 2\pi l, \qquad \int_{C_{\rho, l}} HK d\mu = 0 \ .\]

\noindent\textbf{On a rotated arc $I^{\theta_1, \theta_2}_{(\tilde x,\tilde y), \rho}\subset \bb{R}_+$ centred at $(\tilde x,\tilde y)$ through angles $[\theta_1, \theta_2]$ from the horizontal:} Parametrising as $\gamma(s) =(\rho\cos(\rho^{-1} s)+\tilde x,\rho\sin(\rho^{-1} s)+\tilde y)$  for $s\in[\rho \theta_1,\rho\theta_2]$, we have that $\kappa=\rho^{-1}$ so 
\begin{flalign*}\int_{I^{\theta_1, \theta_2}_{(\tilde x,\tilde y), \rho}} H &= 2\pi\rho^{-1} \int_{\rho\theta_1}^{\rho\theta_2} \rho\cos(\rho^{-1} s)+\tilde x ds+2\pi\rho(\sin(\theta_2)-\sin(\theta_1))\\
&=4\pi\rho[\sin(\t_2)-\sin(\t_1)]+2\pi \tilde x(\t_2-\t_1)
\end{flalign*}
while
\begin{flalign*}
\int_{I^{\theta_1, \theta_2}_{(\tilde x,\tilde y), \rho}} HK d\mu &= 2\pi \rho^{-1}\int_{\rho\theta_1}^{\rho\theta_2} \cos(\rho^{-1} s)\rho^{-1}+\frac{\cos^2(\rho^{-1} s)}{\rho\cos(\rho^{-1} s)+\tilde x} ds\\
&= 2\pi \int_{\theta_1}^{\theta_2}\frac{\cos^2(t)}{\rho\cos(t)+\tilde x} dt+2\pi\rho^{-1}[\sin(\t_2)-\sin(\t_1)]\ .
\end{flalign*}
As $\int_{\theta_1}^{\theta_2} \cos^2(t)dt=\frac{1}{4}[\sin(2\t_2)-\sin(2\t_1)] +\frac 1 2[\t_2-\t_1]$ we may estimate the integral as
\[\pi\frac{\frac{1}{2}[\sin(2\t_2)-\sin(2\t_1)] +[\t_2-\t_1]}{\tilde x+\rho}\leq 2\pi \int_{\theta_1}^{\theta_2}\frac{\cos^2(t)}{\rho\cos(t)+\tilde x} dt \leq \pi\frac{\frac{1}{2}[\sin(2\t_2)-\sin(2\t_1)] +[\t_2-\t_1]}{\tilde x-\rho}\]

 \begin{figure}
 \includegraphics[scale=0.16]{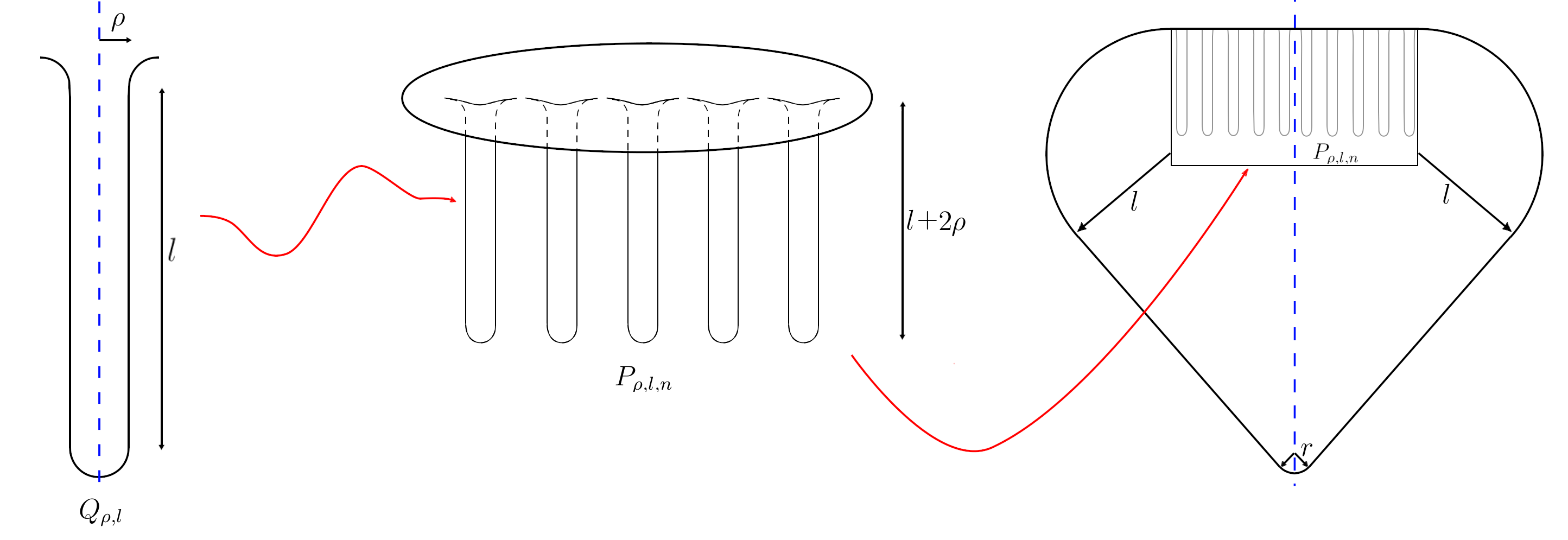}
 \caption{The construction of the Trilobite example in pictorial form}
 \end{figure}
\begin{example}[The Trilobite]\label{ex:trilobite}
We construct a family of embedded surfaces in $\bb{R}^3$ as follows. We start with a construction that is $C^1$ everywhere and smooth away from 1 dimensional glueing points. 
\begin{description}
\item[\textbf{Step 1 - Capped cylinders, $Q_{\rho,l}$}] We start with a hemisphere of radius $\rho$, $E_\rho$, attach this to a long cylinder $C_{\rho, l}$ and then attach this to the plane using a quarter circle $I_{(2\rho,-\rho),\rho}^{(\frac{\pi}{2}, \pi)}$ of radius $\rho$, as shown. We will call this entire union $Q_{\rho,l}$. In our constuction, we will choose out normals to point into the cylinders meaning that they contribute negatively to both integrals. We calculate
\begin{center}
\begin{tabular}{|| c || c | c |}
\hline\hline
&$\int Hd\mu$&$\int HK d\mu$\\
\hline\hline
$E_\rho^1$ &$-4\pi \rho$ & $-4\pi \rho^{-1}$\\
$C_{\rho, l}$& $-2\pi l$&$0$\\
$I_{(2\rho,-\rho),\rho}^{(\frac{\pi}{2}, \pi)}$&$2\pi\rho(\pi-2)$&$\geq\frac{\pi^2}{6\rho} -2\pi\rho^{-1}$\\
\hline\hline
\end{tabular}
\end{center}
 We get that
 \[\int_{Q_{\rho,l}} H d\mu = 2\pi[-(4-\pi)\rho-l], \qquad \int_{Q_{\rho,l}} KH d\mu \geq -6\pi\rho^{-1} \]
\item[\textbf{Step 2 - Attach $n$ capped cylinders to a disk}]  We now attach $n$ of these to a flat disk of radius $2n \rho$ which we may always do (e.g.\ in a line). We will refer to this modified disk as $P_{\rho,l,n}$ where we note that the integrals over $H$ and $HK$ on this disk are just $n$ times those on $Q_{\rho,l}$. 
\item[\textbf{Step 3 - Glue into a final rotationally symmetric surface}]We now attach this to the rotationally symmetric surface given by taking another third of a circle $I^{-\frac{\pi}{6}, \frac{\pi}{2}}_{(2n\rho,-l),l}$ before closing the surface with a cone of slope $\sqrt{3}$ which we will denote $O_{\sqrt{3}}$. We round off the point of the cone using a spherical cap of radius $r$ denoted $A_r$.
\item[\textbf{Step 4 - A careful choice of constants and smoothing}] In the above, altering $l$ in the cylinders $Q_{\rho, l}$ changes only $\int H$, and by increasing $n$ sufficiently, we will see below see that this may be used to ensure that in the above we have $\int H=0$. Unfortunately, the caps and joins adds a significant negative quantity to $\int HK$. However $A_r$ has only negligable $\int H$ contribution while adding an arbitrarily large amount to $\int HK$. Balancing these (see calculations below) mean we can ensure that $\int H=0$ while $\int HK$ can be made arbitrarily large. However, we would like a smooth manifold - to do this we smooth only locally to the joins, perturbing $\int H$, $\int HK$ by an arbitrarily small amount. Away from the joins, the interior of the cylinders are still cylinders and so shortening or lengthening one of these slightly will again restore $\int H=0$.

We now do the accountancy and calculate the following contributions to the integrals:
\begin{center}
\begin{tabular}{|| c || c | c ||}
\hline\hline
&$\int Hd\mu$&$\int HK d\mu$\\
\hline\hline
$P_{\rho,l,n}$ &$2n\pi[-(4-\pi)\rho-l]$ & $\geq-6n\pi\rho^{-1}$\\
\hline
$I^{-\frac{\pi}{6}, \frac{\pi}{2}}_{(2n\rho,-l),l}$& $6\pi l+\frac{8}{3}n\pi^2\rho$&positive\\
\hline
$O_{\sqrt{3}}$ &$2n\rho+l-r$&$0$\\
\hline
$A_r$&$2\pi(2-\sqrt{3}) r$&$2\pi(2-\sqrt{3}) r^{-1}$\\
\hline\hline
\end{tabular}
\end{center}
so
\begin{flalign*}
&\int H d\mu = 2\pi[(3+(2\pi)^{-1}-n)l+(\frac{11}{3}\pi+\pi^{-1} -4) n\rho+(2-\sqrt{3}-(2\pi)^{-1})r]\\
&\int HK d\mu \geq 2\pi(2-\sqrt{3}) r^{-1}-6n\pi\rho^{-1}\ .
\end{flalign*}
Therefore we may pick (for example) $\rho=1$, $n\geq 4$ (e.g.\ $n=7$), and choose 
\[l(r)=\frac{(\frac{11}{3}\pi+\pi^{-1} -4) n\rho+(2-\sqrt{3}-(2\pi)^{-1})r}{n-3-(2\pi)^{-1}}\ .\]
Then, by setting $r$ to sufficiently small (compared to $n^{-1}\rho$) we may make $\int HK d\mu$ arbitrarily large. 
\end{description}
\end{example}

\bibliographystyle{plain}
\bibliography{Lit.bib}

\end{document}